\documentclass[12pt]{article}
\usepackage{amssymb,amsthm,amsmath,amsfonts,latexsym,tikz,hyperref,start4,ytableau,authblk}
\usepackage[hmargin=1in,vmargin=1in]{geometry}
\usepackage{comment}

\newtheorem{thm}{Theorem}[section]
\newtheorem{prop}[thm]{Proposition}
\newtheorem{cor}[thm]{Corollary}
\newtheorem{lem}[thm]{Lemma}

\newtheorem{question}[thm]{Question}
\newtheorem{rem}[thm]{Remark} 

\newcommand{\Omb}{\breve{\Om}}
\newcommand{\cEb}{\breve{\cE}}
\newcommand{\compb}{\breve{\comp}}
\newcommand{\da}{\downarrow}
\DeclareMathOperator{\Fix}{Fix}
\DeclareMathOperator{\crk}{crk}
\DeclareMathOperator{\type}{type}
\DeclareMathOperator{\sdiv}{sd} 
\DeclareMathOperator{\Comp}{comp} 

\begin{document}
\pagestyle{plain}

\title{Ordered set partition posets
}
\author[1]{Bruce E. Sagan}
\author[2]{Sheila Sundaram}
\affil[1]{Department of Mathematics, Michigan State University, East Lansing, MI 48824}
\affil[2]{School of Mathematics, University of Minnesota, Minneapolis, MN 55455}

\date{\today\\[10pt]
	\begin{flushleft}
	                                       \\[5pt]
	\end{flushleft}}

\maketitle

\begin{abstract}
A set partition is said to be ordered if the blocks of the partition are listed in a specific order. The ordered set partitions  of $\{1,\ldots,n\}$, with a unique minimal element adjoined, form a lattice $\Om_n$ with respect to refinement. The lattice $\Om_n$ is well known to be  the face lattice of  the permutohedron.  In this paper we study the combinatorics and topology  of two subposets of $\Om_n$ with restricted block sizes, either all  divisible by some fixed $d\ge2$, or all  congruent to $1$ modulo $d$.  For the $d$-divisible case we derive an explicit recursive atom ordering for the lattice, as well as formulas for 
the action of the symmetric group on the Whitney homology and the rank-selected homology, and also for the multiplicity of the trivial representation. In the 1 mod $d$ case we show that the poset has a curious  interval structure related to the $k$-Catalan numbers.
Our investigations lead to  enumerative invariants in both cases.  Open problems and avenues for future research are scattered throughout.
\end{abstract}


\section{Introduction}

For  nonnegative integers $m,n$ with $m\le n$ we will use the notation
\beq
\label{[n]}
[n]=\{1,2,\ldots,n\} \qmq{and} [m,n]=\{m,m+1,\ldots,n\}.
\eeq
If $S$ is any finite set then we will use the notations $\#S$ or $|S|$ for the cardinality of $S$.  

A {\em set partition} $\pi$ of $S$ is a family of nonempty subsets $B_1,\ldots,B_k$ called {\em blocks} such that  $S=\uplus_i B_i$ (disjoint union) and we write $\pi=B_1/\ldots/B_k$.  We may leave out set braces and commas in examples.  Note that the order of the blocks does not matter so that, for example, $134/26/5=26/5/134$.
The partition is {\em $d$-divisible} if $d$ divides $\#B_i$ for all $i$.   Partitions can be partially ordered by {\em refinement} where $B_1/\ldots /B_k\le C_1/\ldots/C_l$ if each $C_j$ is a union of certain $B_i$.  The poset of all partitions of $[n]$ ordered by refinement is a lattice denoted $\Pi_n$. Also, ordering the $d$-divisible partitions and adding a unique minimal element $\zh$ results in a lattice $\Pi_n^{(d)}$.  Notice that the symmetric group $\fS_n$ acts on both these lattices.  The combinatorial, topological, and representation-theoretic properties of $\Pi_n$ and $\Pi_n^{(d)}$  have been extensively studied.  See, for example, \cite{chr:PLMS1986,com:itf,EH:fpl,HS:fcp,HH:2003,su:AIM1994,su:RPS70-2016,wac:bhd,woo:cce}.

An {\em ordered set partition} of a finite set $S$ is a sequence of non-empty subsets $\om=(B_1,B_2,\ldots,B_k)$ with  
$\uplus_i B_i = S$, where the $B_i$ are again called {\em blocks}.  Note the use of parentheses and commas, as opposed to forward slashes, to indicate that now the order of the blocks matters.  
As with ordinary set partitions, we will often leave out the set braces and commas in each $B_i$.  
The ordered set partitions of $[3]$ are displayed in Figure~\ref{Om3}.
We will use letters near the end of the Greek alphabet for ordered partitions.
The number of blocks of $\om$ is its {\em length}, denoted $\ell(\om)$.
The {\em Stirling numbers of the second kind} are 
$$
S(n,k) = \text{the number of unordered partitions of $n$ into $k$ blocks}.
$$
It follows that 
\beq
\label{k!S}
k!S(n,k) = \text{\# of ordered partitions of $n$ into $k$ blocks}.
\eeq
Ordered set partitions and Stirling numbers have connections to combinatorics, number theory, polyhedral theory, and coinvariant algebras.  They appear in a closeted form as far back as a paper of Carlitz~\cite[equation (11)]{car:af}.  See the papers of  Ishikawa, Kasraoui, and Zeng~\cite{IKZ:ems} or Sagan and Swanson~\cite{SS:qsn} for history and references.

As in the unordered case, if $d\ge 1$ is an integer then we say that $\om$ is {\em $d$-divisible}
if $d$ divides $\# B_i$ for all $i$.  In this case we write
$$
\om\comp_d S.
$$
We will drop the $d$ if $d=1$ so that there is no restriction on the block sizes.
To illustrate, the $2$-divisible partitions of $\{a,b,c,d\}$ are
$$
(ab,cd),\ (ac,bd),\ (ad,bc),\ (bc,ad),\ (bd,ac),\ (cd,ab), \text{ and } (abcd).
$$

\begin{figure}
    \centering
\begin{tikzpicture}
\draw(0,0) node{$\zh$};  
\draw(-5,2) node{$(1,2,3)$};
\draw(-3,2) node{$(1,3,2)$};
\draw(-1,2) node{$(2,1,3)$};
\draw(1,2) node{$(2,3,1)$};
\draw(3,2) node{$(3,1,2)$};
\draw(5,2) node{$(3,2,1)$};
\draw(-5,4) node{$(12,3)$};
\draw(-3,4) node{$(1,23)$};
\draw(-1,4) node{$(13,2)$};
\draw(1,4) node{$(2,13)$};
\draw(3,4) node{$(23,1)$};
\draw(5,4) node{$(3,12)$};
\draw(0,6) node{$(123)$};  
\draw (-0,.3)--(-5,1.7) (0,.3)--(-3,1.7) (0,.3)--(-1,1.7)
(0,.3)--(5,1.7) (0,.3)--(3,1.7) (0,.3)--(1,1.7)
(0,5.7)--(-5,4.3) (0,5.7)--(-3,4.3) (0,5.7)--(-1,4.3)
(0,5.7)--(5,4.3) (0,5.7)--(3,4.3) (0,5.7)--(1,4.3)
(-5,2.3)--(-5,3.7) (-5,2.3)--(-3,3.7)
(-3,2.3)--(-3,3.7) (-3,2.3)--(-1,3.7)
(-1,2.3)--(-5,3.7) (-1,2.3)--(1,3.7)
(1,2.3)--(1,3.7) (1,2.3)--(3,3.7)
(3,2.3)--(-1,3.7) (3,2.3)--(5,3.7)
(5,2.3)--(3,3.7) (5,2.3)--(5,3.7)
;
\end{tikzpicture}
    \caption{The poset $\Om_3$}
    \label{Om3}
\end{figure}

We partially order the $d$-divisible partitions of $S$ by insisting that 
$(B_1,B_2,\ldots,B_k)$ is covered by all elements of the form 
$$
(B_1,\ldots,B_{i-1},B_i\uplus B_{i+1},B_{i+2},\ldots B_k) 
\text{ for } 1\le i<k
$$
and extending by transitivity.  In other words, one is permitted to merge adjacent blocks, keeping the new block in the same relative position with the other blocks.
From this one sees that $\om\le\psi$ in this partial order if each block of $\psi$ is a union of adjacent blocks of $\om$, and one block $B$ of $\psi$ is to the left of another $C$ if the blocks of $\om$ contained in $B$ are to the left of those in $C$.
Adding a unique minimal element $\zh$  results in a bounded poset which we will call $\Om_S^{(d)}$.  In the case 
$S=[n]$ we will write this as 
$$
\Om_n^{(d)}=\{\zh\}\ \uplus\ \{\om \mid \om\comp_d [n]\}
$$
where when $n=0$ the second set is considered to be empty.
We will shorten $\Om_n^{(1)}$ to just $\Om_n$.  
And if we write $\om\in\Om_n^{(d)}$ then we are tacitly assuming that $\om$ is an ordered set partition in $\Om_n^{(d)}$, i.e., $\om\neq\zh$.
If we write $x\in\Om_n^{(d)}$ then $x$ could be any element of the poset, including $\zh$ and similarly for other letters near the end of the Latin alphabet.
See Figure~\ref{Om3} for the Hasse diagram of $\Om_3$.

The posets $\Om_n^{(d)}$ have appeared in the literature as tools to help understand other objects such as the permutohedron for the case $d=1$ in the paper of Billera and Sarangarajan~\cite[Proposition 1.4]{bilsar:fpsac1994}, or pointed set partitions as in the article of Ehrenborg and Jung~\cite{EJ:resptn}. 
But, to our knowledge, the $\Om_n^{(d)}$ have not been the primary focus of any previous work.

The purpose of the current article is to study their  combinatorial, topological, and representation-theoretic properties.  The rest of this paper is organized as follows.  
In Section~\ref{cpOmnd} we concentrate on the combinatorics of $\Om_n^{(d)}$.
\begin{itemize}
    \item Theorem~\ref{OmndComb} gives an expression for the M\"obius function of $\Om_n^{(d)}$ in terms of generalized Euler numbers.
    \item We give an explicit recursive atom ordering of the lattice in Theorem~\ref{Om^d-RAO}. 
\end{itemize}
The symmetric group $\fS_n$ acts naturally on $\Om_n^{(d)}$ by permuting the entries of the blocks.  Section~\ref{asg} looks at the induced action on various associated homology groups.  
\begin{itemize}
    \item Theorem~\ref{OmdWhit} shows that the Frobenius characteristic of the action on the Whitney homology and its dual can be expressed in terms of complete homogeneous symmetric functions $h_\la$ where $\la$ is an integer partition with all parts divisible by $d$.
\end{itemize}
In Section~\ref{rs}, rank-selected and corank-selected subposets of $\Om_n^{(d)}$ are considered.  
\begin{itemize}
\item Theorem~\ref{thm:sdP-rksel-recS} offers a recursive  formula for the rank-selected homology of the barycentric subdivision of any  Cohen-Macaulay poset.  We establish the equivalence with a different  description due to  Stanley \cite{sta:gap}.
    \item Theorem~\ref{rksel-Omd} specializes  this  recurrence to the Frobenius characteristic of the corank-selected homology of $\Om_{dm}^{(d)}$.
\end{itemize}
 Section~\ref{mtr} is devoted to studying the multiplicity $b_m(T)$ of the trivial representation of $\fS_{dm}$ acting on the homology of  the subposet of $\Om_{dm}^{(d)}$ selected by coranks $T$.
\begin{itemize}
    \item In Theorem~\ref{TrivRepCom} we show that $b_m(T)$ counts permutations in $\fS_{m-1}$ with descent set $T$.
    \item A similar result, Theorem~\ref{thm:triv-restrict-zeros}, is obtained for the action of $\fS_{dm-1}$ considered as the subgroup of $\fS_{dm}$ which fixes $dm$.  We obtain what appears to be a new refinement of $m!$ by subsets of $[m-1]$.
    \item Proposition~\ref{DeOm-quot} determines the topology of the quotient complex $\Delta(\Om_{dm})/\fS_{dm}$.  
\end{itemize}
Finally, Section~\ref{bsr} considers the combinatorics of the subposet of $\Om_n$
where all the block sizes are congruent to $1$ modulo $d$.  
\begin{itemize}
    \item The M\"obius function is given up to sign by a generalization of the Catalan numbers, the \emph{$k$-Catalan numbers}, as shown in Theorem~\ref{thm:muOmb}. 
    \item Theorem~\ref{bOmndComb} gives generating function  formulas for the M\"obius number, that we call the \emph{remainder 1 Euler numbers}.  
\end{itemize}

\section{Combinatorial properties of $\Om_n^{(d)}$}
\label{cpOmnd}

In this section, we will study the combinatorics of $\Om_n^{(d)}$.  Throughout, we will assume that $d\ge1$ is a divisor of $n\ge0$.  

In our first result we will collect some elementary properties of this poset.
For more information about posets,  including definitions of any undefined terms, see the texts of Sagan~\cite[Chapter 5]{sag:aoc} 
or Stanley~\cite[Chapter 3]{sta:ec1}.
Assume that $(P,\le)$ is a poset.  All posets in this work will be finite without further mention.
If $x,y\in P$ with $x\le y$ then we have the closed interval
$$
[x,y] = \{z \mid x\le z\le y\}.
$$
 
If $P$ has a unique minimal element or unique maximal element then they are denoted $\zh$ 
or $\oh$, respectively.
If $P$ contains both of these elements it is said to be {\em bounded}.
We write $x\lt y$ if $x$ is {\em covered} by $y$ in $P$, that is, $x<y$ and there is no $z$ with $x<z<y$. 
If $P$ has a $\zh$ then the {\em atoms} $a$ of $P$ are the elements covering $\zh$.
The poset is {\em ranked} if, for every $x\in P$, all maximal chains from $\zh$ to $x$ have the same length $\ell$.  In that case, we say that $x$ has {\em rank} $\ell$ and write
$$
\rk x =\ell.
$$
We also define the rank of an interval $[x,y]$ to be
$$
\rk(x,y) = \rk y - \rk x.
$$
If $P$ is ranked and has a unique maximal element $\oh$ then we say that $P$ is {\em graded}.  In this case, the {\em corank} of
$x\in P$ is 
$$
\crk x = \rk \oh -\rk x.
$$
A number of our results will be easier to state in terms of corank rather than rank.

We will also need the {\em Boolean algebra}
$$
\cB_n=\{S \mid S\sbe[n]\}
$$
ordered by inclusion as well as the {\em $d$-divisible Stirling numbers of the second kind} which are
$$
S^{(d)}(n,k) =\text{the number of unordered partitions of $n$ into $k$ blocks all of size divisible by $d$}.
$$

Note that, interestingly, the isomorphism in part (d) of the next result depends only on $k$ (the number of blocks of the partition) and not on $n$ (the sum of the parts).
\begin{lem}
\label{OmGen-d}
If $d$ divides $n$ then the poset $\Om_n^{(d)}$ satisfies the following.
\ben
\item[(a)]  It has $\oh =([n])$.
\item[(b)]  Its atoms are the $\om$ with $\#B=d$ for all blocks $B$ of $\om$.
\item[(c)]  It is ranked.  The rank of  $\om=(B_1,\ldots,B_k)$ is 
$$
\rk\om = n/d-k+1.
$$
The number of $\om$ at corank $k$ is $(k+1)! S^{(d)}(n,k+1)$.
\item[(d)]  For any $\om=(B_1,\ldots,B_k)$ we have
$$
[\om,\oh]\iso \cB_{k-1}.
$$
\item[(e)] For any  ordered partitions $\psi,\om\in\Om_n^{(d)}$ we have
$$
[\psi,\om]\iso \cB_{\rk(\psi,\om)}.
$$

\item[(f)] The poset $\Om_n^{(d)}$ is an atomic lattice but it is not semimodular in general. \hqed

\een
\end{lem}
\begin{proof}
(a)  We need to show that every ordered partition $\om\in\Om_n^{(d)}$ satisfies $\om\le([n])$.  But this clearly follows from the description of the partial order.

\medskip

(b)  Using the description of the partial order again, we have that $\om$ will cover $\zh$ if and only if no block of $\om$ can be written as a disjoint union of two proper $d$-divisible subsets.  But this is equivalent to having all blocks of size $d$.

\medskip

(c)  Consider any maximal chain
$$
C:\zh=\om_0 \lt \om_1 \lt \ldots \lt \om_\ell = \om.
$$
By part (b), $\om_1$ has $n/d$  blocks, all of size $d$.  As we lose a block in each cover $\om_i\lt \om_{i+1}$,  to end with an ordered partition with $k$ blocks we must have
$$
\ell = n/d-k+1.
$$
The result now follows from this discussion and the definition of $S^{(d)}(n,k)$.

\medskip

(d) We will construct an anti-isomorphism $A:\cB_{k-1}\ra [\om,\oh]$.  This will suffice since $\cB_{k-1}$ is self dual. 
Take $\om=(B_1,B_2,\ldots, B_k)$ and number the $k-1$ commas as $1,2,\ldots,k-1$ from left to right to obtain
$$
\om = 
(B_1 \stackrel{1}{,} B_2 \stackrel{2}{,} \ldots, 
B_{k-1} \stackrel{k-1}{,} B_k).
$$
Now given any $S\sbe [k-1]$ we form the ordered partition $A(S)$ by removing any comma not labeled by an element of $S$ and taking the disjoint union of any blocks no longer separated by commas.  For ease of notation, we will often suppress the disjoint union signs and write $B_i B_{i+1}$ in place of $B_i\uplus B_{i+1}$.  For example, if $k=6$ then 
$$
\om=(B_1\stackrel{1}{,} 
B_2 \stackrel{2}{,}  
B_3 \stackrel{3}{,}
B_4 \stackrel{4}{,}
B_5 \stackrel{5}{,}
B_6)
$$
The set $S=\{2,5\}$ gives rise to the ordered set partition
$$
A(S) = (B_1\uplus B_2
\stackrel{2}{,}B_3\uplus B_4\uplus B_5
\stackrel{5}{,}B_6)=
(B_1 B_2
\stackrel{2}{,}B_3  B_4 B_5
\stackrel{5}{,}B_6).
$$
It is easy to check that $A$ is invertible and is a poset anti-isomorphism.

\medskip

(e)  This follows immediately from (d) and the fact that all intervals in a Boolean algebra are Boolean algebras.

\medskip

(f)  Since $\Om_n^{(d)}$ has a $\zh$, to show it is a lattice it suffices to show the existence of a join.  Suppose $\psi,\om\in\Om_n^{(d)}$ where
\begin{align*}
\psi & = (A_1,A_2,\ldots,A_k),\\
\om & = (B_1,B_2,\ldots,B_\ell).
\end{align*}
Find the smallest $i$ such that 
\beq
\label{AB}
A_1\uplus A_2\uplus \ldots \uplus A_i
= B_1\uplus B_2\uplus \ldots \uplus B_j
\eeq
for some $j$.  Such an $i$ exists because if $i=k$ then the union of $A$'s is $[n]$.
Call the union in~\eqref{AB} $C_1$ and removed the corresponding blocks from $\psi$ and $\om$.  Iterate this process to find $C_2$ and so forth.  It is easy to see that
$(C_1, C_2,\ldots C_m)$ is the join of $\psi$ and $\om$.

To prove that $\Om_n^{(d)}$ is atomic, consider any of its ordered set partitions $\om=(B_1,B_2,\ldots,B_\ell)$.  Construct a set of atoms $\cA$ as follows.  Put $A=(A_1,A_2,\ldots,A_{n/d})$ into $\cA$ precisely when 
$A_1\uplus A_2\uplus\cdots\uplus A_i=B_1$
where $di=\#B_1$, and
$A_{i+1}\uplus A_{i+2}\uplus\cdots \uplus A_{i+j}=B_2$
  where $dj=\#B_2$, etc.  It is easy to verify that $\Jn\cA=\om$.

Finally to show that $\Om_n^{(d)}$ is neither upper nor lower semimodular, we use the equivalent conditions in terms of covers.  
Let $m=n/d$ and fix an unordered set partition $\{A_1,A_2,\ldots,A_m\}$ of $[n]$ into blocks of size $d$.
In the upper case, consider the atoms
$\psi=(A_1,A_2,A_3,A_4,\ldots,A_m)$ and $\om=(A_2,A_3,A_1,A_4,\ldots,A_m)$ of $\Om_n^{(d)}$.  Clearly, $\psi\mt\om =\zh$ which they cover.  But 
$\psi\jn\om = (A_1\uplus A_2\uplus A_3,A_4,\ldots,A_m)$ which does not cover $\psi$ or $\om$.  For lower semimodular, one uses the ordered set partitions $\psi'=(A_1,A_2\uplus A_3,A_4,\ldots,A_m)$ and 
$\om'=(A_2,A_1\uplus A_3,A_4,\ldots,A_m)$.
\end{proof}

We now consider the M\"obius function of $\Om_n^{(d)}$.
The {\em M\"obius function} of a finite poset $P$ is a map from the closed intervals $[x,z]$ of $P$ to the integers defined by
$\mu(x,x)=1$ and  either of the two equivalent equations
\begin{align}
\mu(x,z) &= -\sum_{x\le y<z} \mu(x,y) 
\label{muUp}
\\
&= -\sum_{x< y\le z} \mu(y,z)
\label{muDn}
\end{align}
for $x<z$.  If $P$ has a $\zh$ and a $\oh$ we write
$$
\mu(P)=\mu(\zh,\oh).
$$
This function is a far-reaching generalization of the M\"obius function in number theory.

To calculate $\mu(\Om_n^{(d)})$ we will need a generalization of the Euler numbers.
The {\em (ordinary)  Euler numbers}, $E_n$, can be defined  by the generating function
$$
\sum_{n\ge0} E_n\frac{x^n}{n!} = \tan x + \sec x.
$$
These constants have a long and venerable history in combinatorics and number theory.
Now given $d\ge1$ we let $\ze_d$ be a primitive $d$th root of unity.
Define the {\em $d$-divisible Euler numbers}, $\cE_n^{(d)}$, by 
\beq
\label{cEndDef}
\sum_{n\ge0} \cE_n^{(d)} \frac{x^n}{n!} = \frac{d}{e^x+e^{\zeta_d x} + e^{\zeta_d^2 x}+\cdots+e^{\zeta_d^{d-1} x}}
=\frac{1}{1+x^d/d!+x^{2d}/{2d}!+\cdots}.
\eeq
The numbers $\cE_n^{(d)}$ were first considered by Leeming and  MacLeod~\cite{LM:pge} who called them generalized Euler numbers.  They have since  been studied by several people~\cite{ges:cge,KL:cpl,LM:gen,sag:gen}.  See also \cite[p. 64, Eqns. (1.58)-(1.59)]{sta:ec1}. It is not hard to show that 
$$
\cE_n^{(2)} = \case{(-1)^{n/2} E_n}{if $n$ is even.}{0}{if $n$ is odd.}
$$
More generally, $\cE_n^{(d)}=0$ if $d$ does not divide $n$. 

If $\fS_n$ is the symmetric group on $[n]$ then the {\em descent set} of $\pi=\pi_1\pi_2\cdots \pi_n\in \fS_n$ is
$$
\Des\pi =\{i \mid \pi_i>\pi_{i+1}\}.
$$
Let
\beq
\label{Adm}
A_{dm}^{(d)}=\{\pi\in\fS_{dm} \mid \Des\pi = \{d,2d,\ldots,(m-1)d\}.
\eeq
Sagan~\cite[Theorem 3.1]{sag:gen} proved the following.
\bth[\cite{sag:gen}]  
\label{cEndSum}
Suppose $n\ge0$ and $d\ge1$ divides $n$.
\ben
\item[(a)]  We have
$$
\cE_n^{(d)} = \sum_{\om\comp_d\hs{2pt} [n]} (-1)^{\ell(\om)}.
$$
\item[(b)] We have 

\vs{10pt}

\eqqed{
\cE_{dm}^{(d)}= (-1)^n\# A_{dm}^{(d)}.
}
\een
\eth

We will need to associate with an ordered set partiton an ordered integer partition or composition.
A {\em composition} of $n$ is a sequence $\al=(\al_1,\al_2,\ldots,\al_k)$ of positive integers called {\em parts} with $\sum_i \al_i = n$.  In this case we write $\al\comp n$ and call $k=\ell(\al)$ the {\em length} of $\al$.  Any ordered set partition $\om=(B_1,B_2,\ldots,B_k)$ has an associated composition 
called the {\em type} of $\om$ and defined as
$$
\type\, \om = (\#B_1,\#B_2,\ldots,\#B_k).
$$
For example $\type(245,16,3789)=(3,2,4)$.  Given a composition $\al\comp n$, the number of ordered set partitions of that type is clearly a multinomial coefficient
\beq
\label{type=al}
\#\{\om\in\Om_n \mid \type\, \om=\al\} = \binom{n}{\al}=\frac{n!}{\prod_i \al_i!}.
\eeq

The final ingredient we will need is a variant of the notion of poset product introduced by Sundaram \cite[pp.\ 287-288]{su:Jer93}.
Let $(P,\le_P)$ and $(Q,\le_Q)$ be two posets, each having a minimum element.  For any such poset $P$ we let $P^-=P\setm\{\zh\}$.  Then the {\em reduced product} of $P$ and $Q$ is
$$
P\dot{\times} Q =  \{\ (\zh_P,\zh_Q)\ \} \uplus (P^-\times Q^-),
$$
where $\times$ is the usual poset product.
Equivalently, $P\dot{\times} Q$ is the subposet of $P\times Q$ obtained by removing all elements of the form $(\zh,q)$ or $(p,\zh)$.  
The notion of reduced product extends to products of three or more posets in the expected manner.  It is easy to see that if $\om\in\Om_n^{(d)}$ has 
$\type\, \om=(\al_1,\al_2,\ldots,\al_k)$ then
\beq
\label{ZhOmRp}
[\zh,\om] \iso 
\Om_{\al_1}^{(d)} \dot{\times} \Om_{\al_2}^{(d)} \dot{\times}\cdots\dot{\times}\Om_{\al_k}^{(d)}.
\eeq
We will need the following result of Sundaram~\cite[Remark 2.6.1]{su:Jer93}.
\bth[\cite{su:Jer93}]
\label{PxQ}
Let $P,Q$ be posets both having a $\zh$ and a $\oh$.  Then

\vs{10pt}

\eqqed{\mu(P\dot\times Q) = -\mu(P)\mu(Q).}
\eth

We now have everything in place to compute $\mu(\Om_n^{(d)})$.
When $d=1$, the following theorem shows that  $\Om_n$ is Eulerian because it is easy to see from the definition of the Eulerian numbers that $\cE_n^{(1)}=(-1)^n$.  
This observation confirms the known fact that $\Om_n$, 
being the  face lattice of the permutohedron, is an Eulerian poset.
\bth
\label{OmndComb}
In $\Om_n^{(d)}$ where $d$ divides $n$ we have the following M\"obius values.
\begin{enumerate}
\item[(a)]  For $\psi,\om$ ordered partitions with $\psi\le\om$,
$$
\mu(\psi,\om) = (-1)^{\rk(\psi,\om)}.
$$
\item[(b)] For the full poset
$$
\mu(\Om_n^{(d)}) = \cE_n^{(d)}.
$$
\item[(c)] For an ordered partition $\om$ with $\type\om =(\al_1,\al_2,\ldots,a_k)$ we have
$$
\mu(\zh,\om) = (-1)^{k-1} \prod_{i=1}^k \cE_{\al_i}^{(d)}. 
$$
\end{enumerate}
\eth
\begin{proof}
(a)  This follows immediately from Lemma~\ref{OmGen-d}, parts (d) and (e).

\medskip

(b)  Ehrenborg and Jung~\cite{EJ:resptn} studied a simplicial complex $\De_{\vec{c}}$ depending on an integer composition, which they denote by $\vec{c}$, whose last part may be zero.  If $\vec{c}=(d^{n/d})$ then the face lattice of $\De_{\vec{c}}$ is dual to $\Om_n^{(d)}$.  Because of this connection, the current part (b) follows from their Corollary 5.4 and Theorem~\ref{cEndSum} (b) above.

\medskip

(c)  This is an easy consequence of part (b), equation~\eqref{ZhOmRp}, and Theorem~\ref{PxQ}.
\end{proof}

As we have just shown, $\mu(\Om_{dm}^{(d)})$ is, up to sign, the number of  $\pi\in\fS_{dm}$ with 
$\Des\pi=\{d,2d,\ldots,(m-1)d\}$.  By contrast, Stanley showed~\cite{sta:es} that 
$\mu(\Pi_{dm}^{(d)})$, where $\Pi_{dm}^{(d)}$ is the (unordered) partition lattice, is a signed version of the number of $\pi\in\fS_{dm-1}$ with the same descent set.  

\begin{question}
Does one of these two results about the M\"obius function imply the other?
\end{question}

We will now consider the topology and shellability of $\Om_n^{(d)}$. We refer to \cite{Bj:ShellCM1980, BW:lsp, sta:gap} for key definitions.  The face lattice of a convex polytope is always CL-shellable~\cite[Theorem 4.5]{BW:lsp}.
As mentioned earlier, 
$\Om_n$ is the 
face lattice of the permutohedron, but there is no such polytope for $d>1$.  

However, there is a connection with the Boolean lattice $\cB_{dm}$.  The rank-selected homology of the shellable and Cohen-Macaulay Boolean lattice $\cB_n$ is well known  from classical results of Solomon and Stanley \cite{sta:gap}.  Define $\cB_{dm}^{(d)}$ to be the rank-selected subposet consisting of subsets of cardinality divisible by $d$. 
We discuss rank-selection in more detail in the next section, but first we elucidate the connection between $\cB_{dm}^{(d)}$ and  $\Om_{dm}^{(d)}$.  

Given a bounded poset $P$,  let $\bar{P}=P\setminus\{\hat 0, \hat 1\}$, the proper part of $P$. The   order complex 
of $P$  is the simplicial complex $\Delta(\bar{P})$ whose faces are the chains $(\hat 0<x_1<\cdots< x_r<\hat 1)$ of $\bar{P}$, $x_i\in P$.  The chain $(\hat 0<\hat 1)$ corresponds to the empty face.  The face poset $\cL(\Delta(\bar{P}))$ of the order complex is thus the poset of chains $(\hat 0<x_1<\cdots< x_r<\hat 1)$ of $\bar{P}$. 

Now let $P$ be the bounded rank-selected subposet $\cB_{dm}^{(d)}$ of $\cB_{dm}$.
As in \cite[Ex. 3.18.9]{sta:ec1}, 
we have a bijection mapping a chain $(\hat 0=\emptyset\subset X_1\subset\cdots\subset X_r\subset\hat 1)$ in the poset $\cB_{dm}^{(d)}$ to the ordered partition $\omega$, where the $r+1$ blocks of $\omega$ are $B_1=X_1, B_2=X_2\setminus X_1, \ldots , B_{r+1}=[dm]\setminus X_r$.  Since every cardinality $|X_i|$ is a multiple of $d$, $\omega\in \Omega_{dm}^{(d)}$.    The chains are the simplices in the order complex $\Delta(\overline{\cB_{dm}^{(d)}})$ of the poset $\cB_{dm}^{(d)}$, so the poset of chains is  the face lattice of the order complex $\Delta(\overline{\cB_{dm}^{(d)}})$.  The empty chain $(\hat 0, \hat 1)$ is mapped to $[dm]$, so in fact this bijection shows that the face lattice of the order complex $\Delta(\overline{\cB_{dm}^{(d)}})$ is poset-isomorphic to the dual 
${\Omega_{dm}^{(d)}}^*$ of the lattice $\Omega_{dm}^{(d)}$ (appending a $\hat 1$ in both cases):
\begin{equation}\label{eqn:iso-Omd-Boolean-md}
\cL(\Delta(\overline{\cB_{dm}^{(d)}})\iso {\Omega_{dm}^{(d)}}^*.
\end{equation}
In particular, because the order complex of the face lattice of $\Delta({P})$ is the barycentric subdivision $\mathrm{sd}\,\Delta({P})$ of $\Delta({P})$ (see, e.g., \cite[\S 8]{sta:gap}), this gives a
homeomorphism of simplicial complexes 
\begin{equation}\label{eqn:barycen}
\mathrm{sd}\, \Delta(\cB_{dm}^{(d)}) \simeq \Delta({\Omega_{dm}^{(d)}}^*)
\simeq \Delta({\Omega_{dm}^{(d)}}). 
 \end{equation}

Rank-selection preserves shellability by \cite[Theorem 4.1]{Bj:ShellCM1980}, and the face lattice of a shellable simplicial complex admits a recursive coatom ordering by \cite[Theorem 4.3]{BW:lsp}.  Putting these facts together, the  poset isomorphism~\eqref{eqn:iso-Omd-Boolean-md} establishes the important fact 
that $\Om_{dm}^{(d)}$ admits a recursive atom ordering, and is therefore (CL-)shellable and Cohen-Macaulay for all $d\ge 1$.

Now CL-shellability is equivalent to having a recursive atom ordering~\cite[Theorem 3.1]{BW:lsp}, defined as follows.
Let $\cA(P)$ be the atoms of a finite, ranked poset with a $\zh$ and a $\oh$.  A linear ordering $a_1,a_2,\ldots,a_t$ of $\cA(P)$ is a {\em recursive atom ordering} or {\em RAO} if
\ben
\item[(R1)]  For all $j$, the interval $[a_j,\oh]$ admits an RAO where the atoms coming first are those covering some $a_i$ for $i<j$
\item[(R2)]  For all $i<k$, if $a_i,a_k< y$ for some $y$ then there exists $a_j$ with $j<k$ and an $x\in P$ with 
$$
a_j,a_k \lt x \le y.
$$
\een

We wish to give an explicit RAO for $\Om_n^{(d)}$.  This RAO is particularly nice because it is an extension of the well-known lexicographic order. 
Interestingly, this strategy does {\em not} give an RAO when we consider the case of block sizes congruent to $1$ modulo $d$ in Section~\ref{bsr}.
Sagan~\cite[Lemma 3]{sag:ses} proved the following lemma which will be useful.
\ble[\cite{sag:ses}]
\label{SemiRAO}
Suppose $P$ is a poset such that $[a,\oh]$ is a semimodular lattice for all $a\in\cA(P)$.  Then P admits an RAO if and only if some ordering of the atoms of $P$ satisfies condition (R2) above.
\ele

\bth\label{Om^d-RAO}
If $d$ divides $n$ then the lexicographic order on the atoms of $\Om_n^{(d)}$ is an RAO.  Thus  $\Om_n^{(d)}$ is both CL-shellable  and Cohen-Macaulay.
\eth
\bprf
The Boolean algebra is a semimodular lattice.  So,
 by combining Lemma~\ref{OmGen-d} (d) and Lemma~\ref{SemiRAO}, it suffices to show that some ordering of the atoms of $\Om_n^{(d)}$ satisfies condition (R2) in the definition of an RAO. The proof now falls into two cases.  The first is similar to what happens when $d=1$, so we will begin by demonstrating that $\Om_n$ satisfies (R2).

  By Lemma~\ref{OmGen-d} (b), if $d=1$ then each atom has the form
$a=(p_1,p_2,\ldots,p_n)$ and so can be identified with the permutation
$p_1 p_2\ldots p_n$. 

We claim that the lexicographic order  $\le_l$ on permutations satisfies (R2).
Take any ordered set partition $\om$ and any atoms
$a_i, a_k<\om$ with $a_i<_l a_k$ so that $i<k$.  Write
\begin{align*}
a_i &= (p_1,p_2,\ldots,p_n),\\   
a_k &= (r_1,r_2,\ldots,r_n).
\end{align*}
Let $s$ be the first index such that $p_s\neq r_s$.  
Since $a_i$ is lexicographically smaller than $a_k$ we must have $p_s<r_s$.
Also, these two elements are in the same positions in their respective atoms and so they must be in the same block $B$ of $\om$.  

Now let
$p_l, p_{l+1},\ldots, p_m$ and $r_l,r_{l+1},\ldots,r_m$ be the elements of $B$ listed as they appear in $a_i$ and $a_k$, respectively.
By minimality of $s$, the element $p_s$ must appear after $r_s$ in $a_k$.
Since $p_s<r_s$, there must be a descent in the permutation
$r_s, r_{s+1},\ldots r_t$ where $r_t=p_s$.  In other words, there is an index $u\in[s,t-1]$ such that $r_u>r_{u+1}$.
Let $a_j$ be $a_k$ with elements $r_u$ and $r_{u+1}$ switched.  So, by the previous inequality, $a_j\le_l a_k$.  Furthermore 
$$
\psi:=a_j\jn a_k=(r_1,\ldots,r_{u-1}, r_u r_{u+1}, r_{u+2},\ldots, r_n)
$$
so that $a_j,a_k\lt \psi$.  Finally, since $\psi$ is formed by merging two elements in the same block $B$ of $\om$, we have $\psi\le \om$, verifying (R2)) for $d=1$.

 Now let $d$ be arbitrary and  consider an atom $A=(B_1,B_2,\ldots, B_{n/d})$ where, by Lemma~\ref{OmGen-d} (b), 
 $\#B_i=d$ for all $i$.
Write the elements of each $B_i$ in decreasing order and let $B_i\le_l B_j$ if $B_i$ is less than or equal to $B_j$ lexicographically.  Finally, let $A\le_l A'$ if the first blocks in which they differ are $B$ and $B'$, respectively, where $B\le_l B'$.

Suppose that $A_i,A_k\le\om$ where $i<k$ and
\begin{align*}
A_i &= (P_1,P_2,\ldots,P_{n/d}),\\
A_k &= (R_1,R_2,\ldots,R_{n/d}),
\end{align*}
Letting $s$ be the first index in which $A_i$ differs from $A_k$, we must have 
$P_s<_l R_s$.  Let $p\in  P_s$ and $r\in R_s$ be the largest elements in which they differ.  This  forces $p<r$.

Since $A_i,A_k<\om$, there must be a block $B\in\om$ with $P_s,R_s\sbe B$.  Let
$P_\ell,P_{\ell+1},\ldots,P_m$ and $R_\ell,R_{\ell+1},\ldots,R_m$ be the blocks of $A_i$ and $A_k$, respectively, which are subsets of $B$ where $\ell\le s\le m$.  Since $p\not\in R_s$ and we are comparing atoms lexicographically, there must be a $t>s$ with $p\in R_t$.  If there is a lexicographic descent in the sequence $R_s,R_{s+1},\ldots,R_t$ then we proceed as in the proof of the $d=1$ case above.

Otherwise, we have $R_s<_l R_{s+1}<_l \ldots<_l R_t$.  Letting $m_i=\max R_i$ for all $i$, this implies that $m_s<m_{s+1}<\ldots<m_t$.  Combining this with the fact that $p<r\in R_s$ gives us $p<m_s\le m_{t-1}$.  Now consider the sets
\begin{align*}
R_t' &= (R_t -\{p\})\cup\{m_{t-1}\},\\
R_{t-1}' &= (R_{t-1} -\{m_{t-1}\}) \cup \{p\}.
\end{align*}
Note that since $R_{t-1}'$ was constructed by removing the maximum of $R_{t-1}$ and replacing it with a smaller element, we must have $R_{t-1}'<_l R_{t-1}$.  Finally, let $A_j$ be the atom obtained from $A_k$ by replacing $R_{t-1}$ and $R_t$ by $R_{t-1}'$ and $R_t'$, respectively. From the inequality on the $(t-1)$st blocks, it follows that we must have 
$A_j <_l A_k$.  Also, the fact that
$R_t\cup R_{t-1} = R_t'\cup R_{t-1}'$ shows that $A_j < \om$.  Now, $A_j$ and $A_k$ only differ in two adjacent blocks so that $A_j,A_k\lt \psi$ for some $\psi$.
Finally, since $A_j,A_k<\om$ we are forced to have $\psi=A_j\jn A_k\le \om$, finishing the proof.
\eprf

One could hope that $\Om_{dm}^{(d)}$ has an even stronger lexicographic shelling property. 

\begin{question}
 Does $\Om_{dm}^{(d)}$ have an EL-labeling? 
\end{question}

\section{The action of the symmetric group}
\label{asg}

In this section, we will study the action of the symmetric group $\fS_n$ on  various homology groups, over the rationals, associated with $\Om_n^{(d)}$.  For more about the theory of symmetric group representations see the books of James~\cite{jam:rts},  James and Kerber~\cite{JK:rts}, Sagan~\cite{sag:sg}, or 
Serre~\cite{se:1977}.  For poset topology we refer to \cite{sta:gap} and \cite[Chapter 3]{sta:ec1}, and to Macdonald~\cite{Macd:1995}, Sagan~\cite{sag:sg}, or Stanley~\cite{sta:ec2} for symmetric functions. 

Let $P$ be a bounded poset. The {\em proper part} of $P$ is defined as 
$
\Pb := P-\{\zh,\oh\}.
$
Recall that the set of all chains in $\Pb$ forms a simplicial complex called the {\em order complex} of $P$ and denoted $\De(P)$.  We write the $i$th (reduced) homology group of $\De(P)$ over the rationals as $\Ht_i(P)$.  

Suppose that $P$ is Cohen-Macaulay. Then $P$ is graded, so suppose $\rk\oh=r$.  In this case, all its homology groups vanish except possibly in the top dimension, $r-2$, and $\Ht_{r-2}(P)$ is a vector space of dimension $|\mu(P)|$.   
Furthermore, any group of automorphisms of $P$ induces an action on $\Ht_{r-2}(P)$.

The action of the symmetric group $\fS_n$   is most conveniently described using its Frobenius characteristic and symmetric functions.  Writing $\ch$ for the Frobenius characteristic,  we set 
\begin{equation}
  \label{betaDef} 
  \beta_{dm}^{(d)} := \ch\tilde{H}_{m-2}(\Omega_{dm}^{(d)}).
\end{equation}
  We will also need the complete homogeneous functions $h_n$.  As usual, if $\la=(\la_1,\ldots,\la_k)$ is a partition,  we let 
  $
  h_\la = h_{\la_1}\cdots h_{\la_k}.
  $
  We extend this definition to compositions $\al=(\al_1,\ldots,\al_k)$ by letting 
  $
  h_\al = h_{\al_1}\cdots h_{\al_k}.
  $
  We use the same conventions for the elementary symmetric functions $e_\la$ and $e_\al$ as well as the Frobenius characteristics $\be_\la^{(d)}$ and $\be_\al^{(d)}$, where in the last two cases $\la$ and $\al$ must have all parts divisible by $d$.
  Note that, although we use parentheses for both partitions and compositions, context should make it clear which is meant. Finally, if $d$ is a positive integer and $\al=(\al_1,\ldots,\al_k)$ is a composition then we let
  $$
  d\al=(d\al_1,\ldots,d\al_k)
  $$
  and similarly for partitions $d\la$.

It will be convenient to have a notation for the set of compositions of a given integer of a given length.  We let
$$
\cC(n,k) = \{ \al\comp n \mid \ell(\al) = k\}.
$$

Consider the poset $\Om_{dm}^{(d)}$.  As already observed, 
the poset isomorphism~\eqref{eqn:iso-Omd-Boolean-md} shows that $\Om_{dm}^{(d)}$ is Cohen-Macaulay.
%
Now $\fS_{dm}$ is a group of automorphisms of $\Om_{dm}^{(d)}$ and so acts on its homology.  
Observe that the isomorphism~\eqref{eqn:iso-Omd-Boolean-md} 
  commutes with the action of $\mathfrak{S}_{dm}$. 
 The homeomorphism in~\eqref{eqn:barycen} is thus also group-equivariant,
 and hence we have  
 an $\mathfrak{S}_{dm}$-isomorphism of homology modules 
$
\tilde{H}(\cB_{dm}^{(d)})
\iso \tilde{H}(\Omega_{dm}^{(d)}).
$

Using well-known results about rank-selection in the Boolean lattice \cite{sta:gap},  described in more detail in Theorem~\ref{BooleanRs}, Proposition~\ref{SkewStairNew} below now follows. 
Our conventions for Ferrers diagrams of integer partitions and standard Young tableaux follow  \cite{Macd:1995, sag:sg, sta:ec2}, increasing left to right along rows, and increasing top to bottom down the columns. 
A rim hook (also called a border strip or a ribbon) \cite{Macd:1995, sta:ec2} is a skew shape  whose Ferrers diagram has the property that consecutive rows share exactly one column.  Given a rim hook $\rho$, we let $s_\rho$ be the corresponding Schur function. 
\begin{prop}
\label{SkewStairNew} 
Let $\rho_{d^m}$ be the rim hook consisting of $m$ rows of length $d$.
Then 
$$
\beta_{dm}^{(d)}=s_{\rho_{d^m}}.
$$
In particular, 
$$
\dim\tilde{H}_{m-2}(\Omega_{dm}^{(d)}) = (-1)^m \mathcal{E}_{dm}^{(d)}.
$$
\end{prop}
See \cite[Theorem 4.2]{EJ:resptn} for a more general result, of which this is the special case $\vec{c}=(d^m)$.
For all $d$, we know that intervals in $\Om_{dm}^{(d)}$ are reduced products by~\eqref{ZhOmRp}.  The following  fact about how the homology of a reduced product behaves under group actions was established 
in~\cite[Proposition 2.5, Proposition 2.6]{su:Jer93}.

\begin{prop}
[{\cite{su:Jer93}}]
\label{EqRp} Let $P_1$ and  $P_2$ be Cohen-Macaulay posets of ranks
$r_1$ and $r_2$, respectively, Let $G_i$ be a finite group of automorphisms of $P_i$ for $i=1,2$. Then  $P_1\dot{\times} P_2$
 is Cohen-Macaulay, and there is a $(G_1\times G_2)$-isomorphism

 \vs{10pt}

 \eqqed{
 \tilde{H}_{r_1-2}(P_1)\otimes \tilde{H}_{r_2-2}(P_2)\iso
 \tilde{H}_{r_1+r_2-3}(P_1\dot{\times} P_2). 
 }
\end{prop}

We will need the above result for the 
 next task in this section, which  is to write down the $\fS_{dm}$-module structure of the  
 Whitney homology modules of $\Om_{dm}^{(d)}$, defined below.

The Whitney homology $W\!H_i(P)$ of a poset $P$ was originally defined by Baclawski \cite{bac:Whit1975}. The following
 equivalent definition of Whitney homology, due to Anders Bj\"orner, was shown to be  useful for determining group actions on poset homology by Sundaram~\cite{su:AIM1994}. 
Let $P$ be a ranked poset with least element $\hat 0$, and let $r$ be the length of a longest chain in $P$.  The $i$th Whitney homology of $P$ is related to the usual order homology of $P$ by isomorphisms establishing that, for $0\le i\le r$,
$$
 W\!H_i(P)\iso\bigoplus_{\rk(x) = i} \tilde{H}_{i-2}(\hat 0, x). 
 $$
In particular $W\!H_0(P)$ is the trivial module and
$
W\!H_r(P)\iso\tilde{H}_{r-2}(P).
$
These isomorphisms can be shown to commute with any group of  automorphisms of $P$ \cite{su:AIM1994}. 
  Because of this, we will replace $\iso$ with $=$ when dealing with Whitney homology and the corresponding simplicial homology groups.
The following acyclicity property of Whitney homology  leads to an expression for  $W\!H_r(P)$ as an alternating sum of $G$-modules.

\begin{prop}[{\cite[Lemma 1.1]{su:AIM1994}}]
\label{acyc-Whit}  Let $P$ be a bounded poset, and let $G$ be a group of automorphisms of $P$.  Assume that the Whitney homology is free in all degrees. Then each Whitney homology module is a $G$-module, and  as a virtual sum of $G$-modules one has 
$$
W\!H_r(P)-W\!H_{r-1}(P)+\cdots+(-1)^r W\!H_0(P)=0,
$$
where $r$ is the length of a longest chain in $P$.
\end{prop}

It is often useful to examine the Whitney homology of the dual poset. For example, as noted in Lemma~\ref{OmGen-d} (d), in the case of  $\Om_{dm}^{(d)}$, the intervals $[\om,\oh]$  have a simple description as Boolean lattices. 
As in \cite{su:AIM1994}, we define the dual Whitney homology $W\!H^*(P)$ of a graded poset $P$ of rank $r$  to be the Whitney homology of the dual poset $P^*$, so that, for $0\le i\le r$, 
$$
W\!H^*_i(P)=\bigoplus_{\crk(x) = i} \tilde{H}_{i-2}(x,\hat 1).
$$
Again we have that $W\!H^*_0(P)$ is the trivial module and
$  W\!H^*_r(P)\iso\tilde{H}_{r-2}(P). $

Now assume  $P$ is a Cohen-Macaulay poset   of rank $r$.  Then for every open interval $(x,y)$ in $P$, nonvanishing homology can occur only in the top dimension. 
We will often simply write  $\tilde{H}(x,y)$ for that homology group.
Proposition~\ref{acyc-Whit} gives the following two formulas for the  homology $\tilde{H}_{r-2}(P)$.  
\begin{align}
\label{ac-WH-top}
\tilde{H}_{r-2}(P) &= W\!H_{r-1}(P)-W\!H_{r-2}(P)+\cdots+(-1)^{r-1} W\!H_0(P).\\
\label{ac-WH*-top}
\tilde{H}_{r-2}(P) &= W\!H^*_{r-1}(P)-W\!H^*_{r-2}(P)+\cdots+(-1)^{r-1} W\!H^*_0(P).
\end{align}
\begin{thm}\label{OmdWhit}  
\phantom{Consider $\Omega_{dm}^{(d)}$. }
\begin{enumerate}

\item[(a)] The dual Whitney homology of $\Omega_{dm}^{(d)}$ is given, for 
$0\le k\le m-1$, by 
$$
\ch W\!H^*_{k}(\Omega_{dm}^{(d)})
=\sum_{\al\in\cC(m,k+1)} 
h_{d\al},
$$
and in the top dimension by 
$
\ch W\!H^*_{m}(\Omega_{dm}^{(d)})
=\beta_{dm}^{(d)}.
$
In particular, it is a permutation module except in the top dimension.
\item[(b)]
The Whitney homology of $\Omega_{dm}^{(d)}$ is given, for $0\le k\le m-1$, by 
$$
\ch W\!H_{m-k}(\Omega_{dm}^{(d)})
= \sum_{\al\in\cC(m,k+1)} \beta_{d\alpha}^{(d)},
$$
and in the bottom dimension by
$
\ch W\!H_{0}(\Omega_{dm}^{(d)})= h_{dm}. 
$
\end{enumerate}
\end{thm}
\begin{proof}  
(a) 
Note first that, since  $W\!H^*_{0}(P)$ is always the trivial module,
$$
\ch W\!H^*_{0}(\Omega_{dm}^{(d)})=h_{dm}
$$
which agrees with the sum at $k=0$.  So assume $k\ge1$ and
 consider an upper interval $[\om, \hat 1]$ in $\Omega_{dm}^{(d)}$, where $\om=(B_1, \ldots, B_{k+1})$. Then $\om$ has type $d\alpha$ where $\#B_i=d\alpha_i$ for all $i$, and thus $\alpha\vDash m$.
By Lemma~\ref{OmGen-d} (d) we have  $[\om, \hat 1]\iso \cB_k$, so its homology is one-dimensional.  The stabilizer subgroup of $\om$  is the Young subgroup 
\begin{equation}
    \label{YngSub}
G_\om:=\times_{i=1}^{k+1} \mathfrak{S}_{B_i},
\end{equation}
where $\mathfrak{S}_{B_i}$ is the group of permutations on the elements of the block $B_i$. 
One sees that $G_\om$ acts trivially on the homology, because the group fixes all ordered partitions in $[\om, \hat 1]$.  The orbit of such an $\om$ under the $\mathfrak{S}_{dm}$-action consists of all ordered partitions $\psi$ of  type $d\alpha$.  This is a transitive action, and therefore we have 
$$
\bigoplus_{\type\psi=d\alpha} \tilde{H}_{k-2}(\psi, \hat 1) 
=1 \uparrow_{G_\om}^{\mathfrak{S}_{dm}},
$$
which has Frobenius characteristic $\prod_{i=1}^{k+1} h_{d\alpha_i}$.
It follows that, for $1\le k\le m-1,$ 
$$
\ch W\!H^*_{k}(\Omega_{dm}^{(d)})=\sum_{\rk\psi=m-k} \ch \tilde{H}_{k-2}(\psi, \hat 1)
=\sum_{\al\in\cC(m,k+1)} h_{d\alpha},
$$
since elements at corank $k$ have $k+1$ blocks.

\medskip

(b) The formula for dimension $0$ is immediate from the fact  that $W\!H_{0}(\Omega_{dm}^{(d)})$ is the trivial module.

For the other dimensions, we need to examine the lower intervals  $[\hat 0, \om]$. Again, suppose $\om=(B_1, \ldots, B_k)$ with composition type $d\alpha=(d\alpha_1, \ldots , d\alpha_k)$ where $\#B_i=d\alpha_i$ and thus $\alpha\vDash m$. 
By the isomorphism~\eqref{ZhOmRp} we have that 
$$
[\zh,\om] \iso 
\Om_{d\al_1}^{(d)} \dot{\times}\cdots\dot{\times}\Om_{d\al_k}^{(d)}.
$$
Also, from equation~\eqref{YngSub}, we have that $G_\om=\times_{i=1}^k \mathfrak{S}_{B_i}$ is the stabilizer of $\om$, and hence of $[\zh,\om]$.
 By definition~\eqref{betaDef}, each $ \mathfrak{S}_{B_i}$ acts on the homology of the corresponding component $\Omega_{d\alpha_i}^{(d)}$ of the reduced product like the  representation whose Frobenius characteristic is $\beta_{d\alpha_i}^{(d)}$. 

Now we invoke Proposition~\ref{EqRp}. 
It follows, by collecting orbits as in (a), that the action of $\mathfrak{S}_{dm}$ on the orbit of $\om$ is the induced module 
$\otimes_{i=1}^k \tilde{H}(\Omega_{d\alpha_i}^{(d)} )\uparrow_{G_\omega}^{\mathfrak{S}_{dm}}$, 
and hence its Frobenius characteristic is $\beta_{d\alpha}^{(d)}$.  
Since a partition with $k$ blocks has rank $m-k+1$,
we have shown that 
$$
\ch W\!H_{m-k+1}(\Omega_{dm}^{(d)})= 
\sum_{\al\in\cC(m,k)} \beta_{d\alpha}^{(d)},
$$
and reindexing gives the desired result.
\end{proof}

Using ~\eqref{ac-WH-top} and ~\eqref{ac-WH*-top}, 
we can now conclude 
that the  top homology of $\Omega_{dm}^{(d)}$ satisfies the following identities.
\begin{equation}
\label{OmdWH*} 
\beta_{dm}^{(d)}
=\sum_{k=1}^m (-1)^{m-k} 
\sum_{\al\in\cC(m,k)} h_{d\al}.
\end{equation}
\begin{equation}\label{dimOmdWH*}
    \dim \tilde{H}_{m-2} (\Omega_{dm}^{(d)}) = (-1)^m \cE_{dm}^{(d)}
= \sum_{\omega\in \Omega_{dm}^{(d)} }(-1)^{m-\ell(\omega)}.
\end{equation}
\begin{equation}
\label{OmdWH} 
\beta_{dm}^{(d)} 
= (-1)^{m+1} h_{dm} +
\sum_{k=2}^m (-1)^{k} 
\sum_{\al\in\cC(m,k)} \beta_{d\alpha}^{(d)}.
\end{equation}

The equation~\eqref{OmdWH*} is merely the well-known expansion of the Schur function corresponding to the rim hook $\rho_{d^m}$ in the basis of homogeneous symmetric functions, and can be obtained in a variety of ways, for example by the Jacobi-Trudi identity.  See also Gessel \cite{ges:P-ptns}. 
Here we have obtained it as a consequence of Proposition~\ref{acyc-Whit}.


We also note that in the special case $d=1$,
as $\mathfrak{S}_n$-modules, for $0\le k\le n$ the Whitney homology and the dual Whitney homology are related via the sign representation $\sgn_n$ of $\mathfrak{S}_n$ by the equation
\[ W\!H_{n-k}(\Om_n)=\sgn_n \otimes W\!H^*_k(\Om_n).\]


Next we describe the ordinary generating function for the $\be_{dm}^{(d)}$.
Recall the exponential generating function used to define the
$\cE_{dm}^{(d)}$ in~\eqref{cEndDef}.
One can see that it is precisely the exponential specialization in the sense of Stanley~\cite[Proposition 7.8.4]{sta:ec2} of the generating function below.   

\begin{cor}
\label{OmdGfs} 
Fix $d\ge1$. Then we have the generating function
\begin{equation}\label{OmdGf}
\sum_{m\ge 0} (-1)^m \beta_{dm}^{(d)} x^{dm}
= \frac{1}{1+h_d x^d + h_{2d} x^{2d} + \cdots}.
\end{equation}
In particular, the symmetric functions 
$\{\beta_{dm}^{(d)} \mid m\ge 1\}$ form an algebraically independent set of generators for the ring of symmetric functions generated by 
$\{h_{dm} \mid  m\ge 1\}$.
\end{cor}
\bprf
For any formal power series $1+c_1 x + c_2 x^2 + \cdots$ we have
\begin{align*}
\frac{1}{ 1+c_1 x + c_2 x^2 + \cdots}
 &=\frac{1}{ 1-(-c_1 x -c_2 x^2 -\cdots)}\\
 &=\sum_{k\ge0} (-c_1 x -c_2 x^2 -\cdots)^k\\
 &=\sum_{k\ge0} (-1)^k \sum_{m\ge0} x^m 
 \sum_{\al\in\cC(m,k)} c_{\al_1}\cdots c_{\al_k},
\end{align*}
where the last equality comes from the fact that a term involving $x^m$ in the expansion of $(c_1 x +c_2 x^2 +\cdots)^k$ comes picking the term 
$c_{\al_i}x^{\al_i}$ from the $i$th factor for $1\le i\le k$.
We now get the desired generating function by substituting $x^d$ for $x$,
$c_i=h_{di}$ for $i\ge1$, and using equation~\eqref{OmdWH*}.
\eprf

There is a striking resemblance between the $\mathfrak{S}_{dm}$-action on the top homology of $\Omega_{dm}^{(d)}$ and the $\mathfrak{S}_{dm-1}$-action on the top homology of the (unordered) $d$-divisible partition lattice $\Pi_{dm}^d$. In particular, compare~\eqref{OmdGf} above with~\eqref{dDivRes} below.     The next result was derived in~\cite[Example 1.6(ii), Proposition 5.2]{su:AIM1994}. Equation~\eqref{dDivRes} is originally due to \cite{chr:PLMS1986}, implicitly and in a different form.
Let $\pi_m^d$ denote the Frobenius characteristic of the $\mathfrak{S}_{dm}$-action on the top homology of the $d$-divisible lattice $\Pi_{dm}^d$.

\begin{prop}[{[\cite{chr:PLMS1986}, \cite{su:AIM1994}]}] 
We have the plethystic identity
$$
\sum_{m\ge 0} (-1)^m \pi_m^d 
=\sum_{i\ge 1} (-1)^i \pi_i \left[\sum_{j\ge 1}h_{dj}\right].
$$
We also have the generating function
\begin{equation}
\label{dDivRes} 
\sum_{m\ge 0} (-1)^m  (\pi_m^d)\downarrow_{\mathfrak{S}_{dm-1}} =\frac{\sum_{j\ge 1}h_{dj-1}}{\sum_{j\ge 0} h_{dj}},
\end{equation}
where  $(\pi_m^d)\downarrow_{\mathfrak{S}_{dm-1}}$ denotes restriction of the representation $\pi_m^d$ to $\fS_{dm-1}$.
Moreover, the $\mathfrak{S}_{dm-1}$-representation $(\pi_m^d)\!\downarrow_{\mathfrak{S}_{dm-1}}$ is the skew Schur function indexed by the rim hook whose top row has length $d-1$, and whose remaining $m-1$ rows have length $d$.  
\hqed
\end{prop}
An examination of  the $\mathfrak{S}_{dm}$-action on the maximal chains leads to  the following.
\begin{prop}
\label{MaxChains} The action of $\mathfrak{S}_{dm}$ on the maximal chains of $\Om_{dm}^{(d)}$ has Frobenius characteristic 
\[(m-1)!\, h_d^m.\]
In particular, the number of maximal chains is
$$
(m-1)!\, \binom{dm}{d, \ldots,d}.
$$
\end{prop}
\begin{proof} Let $\omega=(B_1,\ldots,B_m)$ be an atom in $\Om_{dm}^{(d)}$.  Since   $\omega$ has $m$ ordered blocks each of size $d$, its stabilizer is the Young subgroup 
$$
G_\om:=\mathfrak{S}_{B_1}\times\cdots\times\fS_{B_m}\iso \fS_d^m. 
$$
Also, because of the isomorphism in Lemma~\ref{OmGen-d} (d), there is a bijection between the maximal chains from $\om$ to $\oh$ and the maximal chains from $\zh$ to $\oh$ in the Boolean algebra $\cB_{m-1}$.  It follows that there are $(m-1)!$ such chains.

The action of $\mathfrak{S}_{dm}$ permutes the chains from an atom to $\oh$ in $\Om_{dm}^{(d)}$.  Also, because the blocks are ordered,
all chains corresponding to a given chain in $\cB_{m-1}$ are in one orbit of $\mathfrak{S}_{dm}$.
The vector space of maximal chains therefore decomposes into a direct sum of $(m-1)!$ orbits, and each orbit is isomorphic to the induced module 
$1\uparrow_{G_\om}^{\mathfrak{S}_{dm}}.$  
This immediately gives the desired Frobenius characteristic and chain count.
\end{proof}

We close this section with an important property of Whitney homology that we will need when we discuss the rank-selected homology of the barycentric subdivision of a Cohen-Macaulay poset. 
%
%
The following result is crucial to the proof of  Theorem~\ref {thm:sdP-rksel-recS} in the next section. 

\begin{thm}[{\cite[Proposition 1.9]{su:AIM1994}}]\label{thm:rksel-P}  Let $P$ be a Cohen-Macaulay poset of rank $n$ and let $G$ be a group of automorphisms of $P$.
Let $P(\underline{r})$ denote the subposet of $P$ consisting of the bottom $r$ ranks, i.e., the rank set $\{1,2,\ldots , r\}$, for $1\le r\le n-1$. 
Dually, 
 let $P(\bar r)$ denote the subposet of $P$ consisting of the top $r$ ranks, i.e., the corank set $\{1,2,\ldots , r\}$, for $r\le n-1$.
Then one has the following $G$-module isomorphisms in homology:
\begin{equation}\label{eqn:S-AIM-rnksel}\tilde{H}(P(\underline{r})) \bigoplus \tilde{H}(P(\underline{r-1}))
=W\!H_r(P)=\bigoplus_{\substack{x\in P\\ \rk x=r}} \tilde{H}(\hat 0, x).
\end{equation}
\begin{equation}\label{eqn:S-AIM-cornksel}\tilde{H}(P(\overline{r})) \bigoplus \tilde{H}(P(\overline{r-1}))
=W\!H_r^*(P)=\bigoplus_{\substack{x\in P\\ \crk x=r}} \tilde{H}(x, \hat 1).
\end{equation}
\end{thm}


\section{Rank-selection}
\label{rs}

Now we turn to rank-selection.  For any bounded and graded poset $P$ of rank $r$, 
the {\em trivial ranks} are rank $0$ and rank $r$.  The same terminology applies to coranks.
For any subset $S$ of the nontrivial  ranks $\{1, \ldots, r-1\}$,  we define the rank-selected subposet $P(S)$ to be the bounded and graded poset  
\[P(S)=\{x\in P:\rk(x)\in S\} \cup \{\hat 0, \hat 1\}.\]
We will sometimes use the same notation for a corank-selected subposet, but will make it clear when that is being done instead.
  It is known that rank-selection preserves the Cohen-Macaulay property \cite[Theorem 6.4]{bac:1980}.  For a subset of ranks $S\subseteq[m-1]$, write $\beta_{dm}^{(d)}(S)$ for the Frobenius characteristic of  
  $\tilde{H}(\Om_{dm}^{(d)}(S))$.  

Let $P$ be an arbitrary Cohen-Macaulay poset of rank $r$ and let  $S$ be a subset of the ranks $\{1,\ldots, r-1\}$.  Let $G$ be a group of automorphisms of $P$, and let $\alpha_P(S)$ and $\beta_P(S)$ denote, respectively, the $G$-modules arising from the action of $G$ on the maximal chains of $P(S)$ and on the top homology of $P(S)$. 
Then from \cite{sta:gap} one has, as virtual $G$-modules, 
\begin{equation}\label{alpha-beta}
\beta_P(S)=\sum_{T\subseteq S} (-1)^{|T|-|S|} \alpha_P(T),
\end{equation}
and, for the permutation action on the chains, 
\begin{equation}\label{alpha-from-beta}
\alpha_P(S)=\sum_{T\subseteq S}  \beta_P(T).
\end{equation}

We begin by applying the general rank-selection recursion Theorem~\ref{thm:rksel-P} to the rank-selected homology of barycentric subdivision.

Let $P^{\mathrm{sd}}$ denote the poset of chains in $P$, where $(\hat 0, \hat 1)$ is the empty chain.  We append an artificial top element to make $P^{\mathrm{sd}}$ bounded.   Then the order complex of $P^{\mathrm{sd}}$ is the barycentric subdivision of the order complex of $P$, and hence the two order complexes are homeomorphic. Thus $P^{\mathrm{sd}}$ inherits  the Cohen-Macaulay property from $P$, and any group of automorphisms $G$ of $P$ acts also on $P^{\mathrm{sd}}$.   Stanley \cite[\S 8]{sta:gap} gave an elegant formula for the action of $G$ on the  rank-selected homology  of $P^{\mathrm{sd}}$, in terms of  the action on $P$.  
The relevance of this result comes from the observation~\eqref{eqn:barycen} that 
  the order complex of $\Om_{dm}^{(d)}$ is the barycentric subdivision of the order complex of the $d$-divisible Boolean lattice $\cB_{dm}^{(d)}$.

 In order to state Stanley's result we require some results about the Boolean algebra.

The following result of Solomon \cite{sol:1968} and Stanley \cite{sta:gap}  completely determines the $\mathfrak{S}_n$-action on the rank-selected homology in the Boolean lattice $\cB_n$.  Recall from the previous section our conventions for Ferrers diagrams of integer partitions and standard Young tableaux.
If $T=\{1\le t_1<\cdots<t_r\le n-1\}$ is a nonempty subset of the nontrivial ranks of $\cB_n$, define the following.  
\begin{itemize}
\item
$\rho_T$ is the rim hook (border strip) of size $n$ whose rows, top  to bottom, have lengths 
$n-t_r, t_r-t_{r-1}, t_{r-1}-t_{r-2}, \ldots, t_2-t_1, t_1$, and 
\item $f^{\rho_T}$ is the number of standard Young tableaux of the skew shape $\rho_T$.  This is also the number of permutations in $\mathfrak{S}_n$ with descent set equal to $\{t_1, t_2, \ldots, t_r\}$. 
\end{itemize}
Recall that the rank of an element in the Boolean lattice is its cardinality as a set.
\begin{thm}[{\cite[Theorem 4.3]{sta:gap}}]\label{BooleanRs} Let $T=\{1\le t_1<\cdots<t_r\le n-1\}$ be a subset of the nontrivial ranks of $\cB_n$.  The  representation of $\mathfrak{S}_n$ on the homology of the rank-selected subposet $\cB_n(T)$ is isomorphic to the Specht module indexed by the rim hook $\rho_T$, and hence has dimension $f^{\rho_T}$.
\end{thm}
For example, when the rank set is $\{1, 2, \ldots,j\}$, $1\le j\le n-1$,  the rank-selected homology representation for the Boolean lattice $\cB_n$ is the irreducible indexed by  the partition $(n-j, 1^{j})$.  For the rank set $\{k, k+1,\ldots, n-1\}$, $1\le k\le n-1$, the representation is the irreducible indexed by $(k, 1^{n-k})$.

Let $P$ be a Cohen-Macaulay poset of rank $m$. Then  $P^{\sdiv}$ is also a graded Cohen-Macaulay poset of the same rank $m$ as $P$.  Let $G$ be a group of automorphisms of $P$. 
Following \cite[Theorem 8.3]{sta:gap}, for each $i=0, 1, \ldots, m-1$, define $\eta_i(P)$ to be the $G$-module 
\[ \eta_i(P)=\bigoplus_{\substack{T\subseteq [m-1]\\|T|=i}} \beta_{P}(T).
\]

Also define, for each $i=0, 1, \ldots, m-1$ and each subset $T$ of $[m-1]$, the nonnegative integer 
$c_{i,m}(T)$ to be 
\[c_{i,m}(T)=\#\{\sigma\in \mathfrak{S}_m: \sigma(m)=m-i \text{ and } \Des(\sigma)=T\}.
\]
Stanley's elegant formula is the following.
\begin{thm}[{\cite[Theorem 8.3]{sta:gap}}]\label{thm:sd-rs}  Let $P$ be a Cohen-Macaulay poset of rank $m$.  The rank-selected homology representation of the barycentric subdivision $P^{\mathrm{sd}}$ is given by the formula
\[\beta_{P^{\mathrm{sd}}}(T)=\bigoplus_{i=0}^{m-1} c_{i,m}(T)\, \eta_i(P).
\]
\end{thm}

Observe that when $P=\cB_n$, $\eta_i(\cB_n)$ is the \emph{Foulkes} representation, namely, the sum of the Specht modules indexed by $\rho_T$ for all subsets of ranks $T$ of cardinality $i$, and the Frobenius characteristic $\ch \eta_i(\cB_n)$ is  the sum of the Schur functions indexed by rim hooks $\rho_T$ for all $T$ of size $i$.

We will use the Whitney homology technique of Theorem~\ref{thm:rksel-P} to give an  alternative recursive description of the group action on the rank-selected homology of $P^{\mathrm{sd}}$. Our recursion lends itself more easily to representation-theoretic  computations;  moreover, as we show later, multiplicities of certain irreducibles can also be computed more easily from our formulation.  It  exploits certain key features of the barycentric subdivision $P^{\mathrm{sd}}$, extracted in Lemma~\ref{lem:intervalS-special-action} below.  Such a recursive formulation can therefore be obtained for any posets satisfying the condition of the lemma.  We  address the   equivalence with Stanley's description, Theorem~\ref{thm:sd-rs}, in Subsection~\ref{subsec:equiv-S-RPS}.

Let $P$ be a ranked and bounded poset with a group of automorphisms $G$. Observe that the induced action of  $G$  on the barycentric subdivision $P^{\sdiv}$ has the following property: if $g\cdot x=x$ for $x\in P^{\sdiv}$, then  $g\cdot y=y$ for every $y\in P^{\sdiv}$ such that $y\le x$.   
It is now useful to record the following lemma. 

\begin{lem}\label{lem:intervalS-special-action} Let $P$ be a ranked and bounded poset with a group of automorphisms $G$. Assume the action of the group $G$ satisfies this condition: if $g\cdot x=x$ for $x\in P$, then  $g\cdot y=y$ for every $y\in P$ such that $y\le x$. Let  $\mathbf{1}_{G_x}$ denote the trivial representation of the  stabiliser subgroup $G_x$ of $x$. Then for each $x\in P$ at rank $r$, the following hold.
\begin{enumerate} 
\item[(1)] 
The action of  $G_x$ on the reduced homology of the interval $(\hat 0,x)$  is given by 
$ |\mu(\hat 0,x)| \, \mathbf{1}_{G_x}$, provided the interval $(\hat 0,x)$ has homology concentrated in a single degree.
\item[(2)]  The action of $G_x$ on  the maximal chains in the interval $(\hat 0, x)$ is (always) given by 
$\#\{\text{maximal chains in $(\hat 0, x)$}\} \, \mathbf{1}_{G_x}$.
\end{enumerate}
\end{lem}
\begin{proof} In each case the stabiliser subgroup $G_x$ acts trivially on the interval $(\hat 0, x)$, so the representation is simply the trivial module with multiplicity equal to the dimension of the underlying vector space.  For Item (2) this dimension is precisely the number maximal chains from $\hat 0$ to $x$, while for Item (1), since homology occurs in only one degree,   this dimension is given by the absolute value of the (reduced) Euler characteristic, which is $ |\mu(\hat 0,x)|$.
\end{proof}

Item (2) of Theorem~\ref{thm:sdP-rksel-recS} is due to Stanley~\cite[Proposition 8.1]{sta:gap};
it  is included here for completeness.

\begin{thm}\label{thm:sdP-rksel-recS} Let $P$ be a Cohen-Macaulay poset of rank $m$, and let $G$ be a group of automorphisms of $P$. Let $T=\{t_1<t_2 <\cdots < t_r\}$ be a subset of the nontrivial ranks $[m-1]$ of $P$.  Then for the barycentric subdivision $P^{\sdiv}$ we have the following:
\begin{enumerate}
\item[(1)] The $G$-module structure of  
the homology of the rank-selected subposets $\tilde{H}(P^{\sdiv}(T))$ satisfies the recurrence 
\[\beta_{P^{\sdiv}}(T) +\beta_{P^{\sdiv}}(T\setminus \{t_r\})
=
\delta(T) \sum_{\substack{S\subseteq [m-1]\\ |S|=t_r}} \alpha_P(S),
\]
where $\delta(T)=\#\{\sigma\in \fS_{t_r} \mid \Des \sigma= T\setminus \{t_r\} \}$.
\item[(2)]\cite[Proposition 8.1]{sta:gap} The $G$-module structure of  the maximal chains of the  rank-selected subposet $\tilde{H}(P^{\sdiv}(T))$ is  determined by the equality of $G$-modules
\[\alpha_{P^{\sdiv}}(T)= a_P(T) \sum_{\substack{S\subseteq [m-1]\\ |S|=t_r}} \alpha_P(S),
\] 
where $a_P(T)$ is the number of maximal chains $\mathbf{c}_1 \subset \mathbf{c}_2 \subset \cdots \subset \mathbf{c}_r$ with $\rk(\mathbf{c}_i)=t_i, i=1,\ldots r$, and hence $a_P(T)=\binom{t_r}{t_1, t_2-t_1,\ldots, t_r-t_{r-1}}$.
\end{enumerate}
\end{thm}

\begin{proof}   The poset $P^{\sdiv}$ consists of chains $\mathbf{c}=(\hat 0<x_1<\cdots<x_k<\hat 1), x_i\in P $; such a chain is simply a totally ordered $k$-subset  of $\bar P$, making  $P^{\sdiv}$ ranked with rank function defined by the number of elements in the chain, i.e. $\rk(\mathbf{c})=k$.  
A  maximal chain of elements in the rank-selected subposet $P^{\sdiv}(T)$ is thus of the form 
\begin{equation}\label{eqn:chains-in-sdP}\mathbf{c}_1 \subset \mathbf{c}_2 \subset \cdots \subset \mathbf{c}_r, \quad \rk(\mathbf{c}_i)=t_i, \, i=1,\ldots, r.
\end{equation}

From Theorem~\ref{thm:rksel-P}, 
the left-hand side of Item (1) equals the $t_r$th Whitney homology of $P^{\sdiv}$.  That is, we have the equality of $G$-modules 
\begin{equation}\label{eqn:prf-subdivP-hom}
\beta_{P^{\sdiv}}(T) +\beta_{P^{\sdiv}}(T\setminus \{t_r\}) =W\!H_{t_r}(P^{\sdiv})
=\bigoplus_{\substack{\mathbf{c}\in P(T)\\  \rk(\mathbf{c})= t_r}} 
\tilde{H}(\hat 0, \mathbf{c})_{P^{\sdiv}(T)},\end{equation}
where, by definition of $P^{\sdiv}$,  the right-hand side is the sum over all maximal   chains $\mathbf{c}=(\hat 0<x_1<\cdots<x_k<\hat 1) $ in $P$ with $k=t_r$.  

The  subchains  of  a chain $\mathbf{c}$  are simply ordered subsets of the elements of $\mathbf{c}$.
By~\eqref{eqn:chains-in-sdP}, the elements of the interval $(\hat 0, \mathbf{c})_{P^{\sdiv}(T)}$  consist of subchains  of $\mathbf{c}$  with sizes taken from $t_1, \ldots, t_{r-1}$, and hence this interval  is   isomorphic to  the rank-selected subposet $B_{t_r}(T\setminus\{t_r\})$ of the Boolean lattice $B_{t_r}$.  Its homology therefore has dimension 
$\delta(T)$, by Theorem~\ref{BooleanRs}.

Now the action of the stabiliser $G_{\mathbf{c}}$ of the chain $\mathbf{c}$  on the interval $(\hat 0, \mathbf{c})_{P^{\sdiv}}$ clearly satisfies the conditions of Lemma~\ref{lem:intervalS-special-action}, so Item (1) of that lemma applies.  There is one orbit of this action for each  orbit representative of $G$ acting on  chains in $P$  with $t_r$ elements, as described in \eqref{eqn:prf-subdivP-hom}.  But this is precisely the pemutation action on the rank-selected chains of $P$ corresponding to a rank set  $S$ in $P$, i.e., it is the representation $\alpha_P(S)$ for  
$S$ of size $t_r$.   The expression in the right-hand side of the statement in Item (1) follows.

A similar argument using Item (2) of Lemma~\ref{lem:intervalS-special-action} 
establishes the statement about the $G$-action on the rank-selected maximal chains.
\end{proof}

Our recursive formula in Theorem~\ref{thm:sdP-rksel-recS} has the following corollary.

\begin{cor}\label{cor:from-sdP-rksel-recS} With the hypotheses and notations of Theorem~\ref{thm:sdP-rksel-recS}, the sum of homology modules 
\[\beta_{P^{\sdiv}}(T) +\beta_{P^{\sdiv}}(T\setminus \{\max T\})\]
is always a permutation module. Moreover, we have the identity 
\begin{equation}\label{eqn:curious-id}
a_P(T) \left( \beta_{P^{\sdiv}}(T) +\beta_{P^{\sdiv}}(T\setminus \{\max T\})  \right)
=\delta(T)\, \alpha_{P^{\sdiv}}(T).
\end{equation}
\end{cor}

We can replace the poset $P$ with any rank-selected subposet $\cB_n(U)$ of the Boolean lattice.  For the rank-selected homology of its barycentric subdivision, we  then obtain the  equality of $\fS_n$-modules
\[\beta_{\cB_n(U)^{\sdiv}}(T) +\beta_{\cB_n(U)^{\sdiv}}(T\setminus \{\max T\}) 
=\delta(T) \sum_{\substack{S\subseteq U\\|S|=\max T}} \alpha_{\cB_n(U)}(S)
\]
It is well known (e.g., \cite{sag:aoc}) that  subsets $S=\{s_1<s_2<\cdots<s_r\}$ of $[n-1]$ map bijectively to 
compositions $\Comp(S)=(s_1, s_2-s_1,\ldots, n-s_r)$  of $n$, of length $|S|+1$.

When $S\subseteq U$, the $S$-rank-selected subposet of $\cB_n(U)$ coincides with $\cB_n(S)$, and thus $\alpha_{\cB_n(U)}(S)=\alpha_{\cB_n}(S)$; moreover 
$\alpha_{\cB_n}(S) =\prod_{i\in \Comp(S)} h_i$. Hence we obtain the following proposition as a special case of Theorem~\ref{thm:sdP-rksel-recS}. 

\begin{prop}\label{prop:rankSel-Bool-sdiv} Let $U$ be any subset of the nontrivial ranks $[n-1]$ of $\cB_n$, and let $T\subseteq U$.  Let $P$ be the rank-selected subposet $P=\cB_n(U)$.  For the rank-selected homology of its  barycentric subdivision $P^{\sdiv}(T)$, we  obtain the  equality of $\fS_n$-modules 
\[\beta_{\cB_n(U)^{\sdiv}}(T) +\beta_{\cB_n(U)^{\sdiv}}(T\setminus \{\max T\}) 
=\delta(T) \sum_{\substack{S\subseteq U\\|S|=\max T}}  \prod_{i\in \Comp(S)} h_i.
\]
\end{prop}

 The right-hand side of Proposition~\ref{prop:rankSel-Bool-sdiv} can  also be written as a sum over compositions of $m$ of length $1+\max T$ that are coarser than $\Comp(U)$, since  
 $S\subseteq U \iff \Comp(U)$ refines $\Comp(S)$.

\subsection{Equivalence of Theorem~\ref{thm:sdP-rksel-recS} and Theorem~\ref{thm:sd-rs}}\label{subsec:equiv-S-RPS}


In this subsection we show that our rank-selected homology recurrence in Theorem~\ref{thm:sdP-rksel-recS} is in fact equivalent to Stanley's theorem, Theorem~\ref{thm:sd-rs}. 

It suffices to show that, with the hypotheses of Theorem~\ref{thm:sdP-rksel-recS}, for $P$ of rank $m$ and the rank set $T=\{t_1<\cdots< t_r\}$,  we have the equality:

\begin{equation}\label{eqn:Sranksel-equiv-RPS-sd}
\delta(T) \sum_{\substack{S\subseteq [m-1]\\ |S|=t_r}} \alpha_P(S)
= \sum_{i=0}^{m-1}( c_{i,m}(T)+c_{i,m}(T\setminus \{t_r\})\eta_i(P),
\end{equation}
where $c_{i,m}(U)$ and $\eta_i(P)$ are as defined before Theorem~\ref{thm:sd-rs}. 

First we observe that  
\begin{equation}\label{eqn:ref-prf} c_{i,m}(T)+c_{i,m}(T\setminus \{t_r\})= \delta(T) \binom{m-1-i}{m-1-t_r} .
\end{equation}
Indeed, the left-hand side counts the number of permutations $\pi$ in $\fS_m$ such that $\pi(m)=m-i$ and $\Des \pi=T$ or $\Des \pi=T\setminus \{t_r\}$. This forces  $\pi(t_r+1)< \cdots < \pi(m-1) <\pi(m)=m-i$, giving  $\binom{m-1-i}{m-1-t_r}$ choices for  the letters $\pi(t_r+1), \ldots,  \pi(m-1)$. The remaining $t_r$   letters can then be inserted  in $\delta(T)$ ways to create a descent set equal to $T\setminus \{t_r\}$.  Finally $t_r$ itself may or may not be a descent. \footnote{We thank the referee who provided this argument.}

Hence \eqref{eqn:Sranksel-equiv-RPS-sd} becomes 
\begin{equation}\label{eqn:Sranksel-equiv-RPS-sd-2}
 \sum_{\substack{S\subseteq [m-1]\\ |S|=t_r}} \alpha_P(S)
= \sum_{i=0}^{m-1}\binom{m-1-i}{m-1-t_r}\eta_i(P). 
\end{equation}

Using~\eqref{alpha-from-beta}, the left-hand side is  
\begin{align*}\sum_{\substack{S\subseteq [m-1]\\ |S|=t_r}} \sum_{U\subseteq S} \beta_P(U)
&=\sum_{\substack{S\subseteq [m-1]\\ |S|=t_r}} \sum_{i=0}^{m-1} \sum_{\substack{U\subseteq S\\|U|=i}}\beta_P(U)\\
&=\sum_{i=0}^{m-1} \sum_{\substack{U\\|U|=i}}\beta_P(U) \sum_{\substack{ S\supseteq U\\|S|=t_r}} (1)\\
&=\sum_{i=0}^{m-1} \eta_i(P) \#\{S\supseteq U \mid |S|=t_r\}.
\end{align*}

But since $|U|=i$, we have  $\#\{S\supseteq U \mid |S|=t_r\}=\binom{m-1-i}{t_r-i}$ because the set $S$ is obtained by augmenting $U$ with $t_r-i$ elements chosen from the $m-1-i$ elements not in $U$.  We have obtained the right-hand side of~\eqref{eqn:Sranksel-equiv-RPS-sd-2}, completing the proof of the equivalence of our formulation with Stanley's expression in terms of the Foulkes representations.

\subsection{Consequences for $\Om_{dm}^{(d)}$}

We now return to $\Om_{dm}^{(d)}$.  Recall that the face lattice of the order complex of $\bar{\cB}_{dm,d}$ is poset-isomorphic to the dual 
${(\Om_{dm}^{(d)})}^*$ of the lattice $\Om_{dm}^{(d)}$.  Equivalently,  the dual of $\Om_{dm}^{(d)}$ is the barycentric subdivision $P^{\sdiv}$ for the Boolean subposet $P=\cB_{dm}^{(d)}$.
Hence Theorem~\ref{thm:sdP-rksel-recS} translates into the following, with ranks replaced by \emph{coranks} in $\Om_{dm}^{(d)}$:

\begin{thm}
\label{rksel-Omd}
Let $T=\{1\le t_1<\cdots <t_r\le m-1\}$ be a nonempty subset of the \emph{coranks} $[m-1]$ of $\Om_{dm}^{(d)}$.  
Consider the action of the symmetric group $\mathfrak{S}_{dm}$ on  the corank-selected poset $\Om_{dm}^{(d)}(T)$. 
\begin{enumerate}
\item[(a)] The Frobenius characteristic of the action on the top homology of ${\Om_{dm}^{(d)}}(T)$ satisfies
\begin{equation}
\label{RkSelHom} 
\beta_{dm}^{(d)}(T) +\beta_{dm}^{(d)}(T\setminus \{t_r\})
= \delta(T)\sum_{\al\in\cC(m,t_r+1)} h_{d\alpha},
\end{equation}
where 
$$
\delta(T)=\#\{\sigma\in \mathfrak{S}_{t_r} \mid \Des \sigma= \{t_1, \ldots, t_{r-1}\}  \}.
$$
Hence it is a polynomial in 
$\{h_{di} : 1\le i\le m\}$ with  integer coefficients. 
\item[(b)] The Frobenius characteristic of the action on
the maximal chains of $\Om_{dm}^{(d)}(T)$ is
\begin{equation}
\label{RkSelChains}
\al_{dm}^{(d)}(T) =
a_T\sum_{\al\in\cC(m, t_r+1)} h_{d\alpha}
\end{equation}
where 
$$
a_T=\frac{t_r!}{t_1! (t_2-t_1)!\cdots (t_r-t_{r-1})!}.
$$
Hence it is a polynomial in  
$\{h_{di} \mid 1\le i\le m\}$ with nonnegative integer coefficients. 

\end{enumerate}
\end{thm}


We have  the following  pleasing consequence of this theorem and its proof.  Define  the ring homomorphism  \[\Psi_d:h_k\mapsto h_{dk}\] in the ring of polynomials $\mathbb{Z}[h_1, h_2, \ldots]$, in the formal variables $\{h_k\}_{k\ge 0}$.

\begin{cor}\label{cor:Om-to-Omd} Let $f$ be the Frobenius characteristic of the 
$\mathfrak{S}_{m}$-action on the homology of any rank-selected subposet $\Om_m(T)$ of $\Om_m$ (respectively the $\mathfrak{S}_{m}$-action on the maximal chains of any rank-selected subposet of $\Om$).  Write $f$ in the basis of complete homogeneous symmetric functions $\{h_1, h_2,\ldots\}$. Then for any $d\ge 1$, the Frobenius characteristic of the corresponding 
$\mathfrak{S}_{md}$-action on the same modules associated to the poset $\Om_{md}^{(d)}$ is given by $\Psi_d(f)$.
\end{cor}
We end this section with a brief digression  regarding the case $d=1$. 
Since $\Omega_n$ is Eulerian, by Alexander duality \cite[Chapter 3, Section 16]{sta:ec2}, \cite[Section 2]{sta:gap} we have the following observation.
 
  \begin{prop}\label{Alex*-Om} 
  Let $S\subseteq[m-1]$ be a subset of ranks of $\Om_m$, and let $\bar{S}$ be its complement in $[m-1]$. 
  Then 
   \[\tilde{H}(\Om_m(S))\iso  \tilde{H}(\Om_m(\bar{S})) \, \sgn_{\mathfrak{S}_m}.\]
  \end{prop}
  \begin{proof} This follows by Alexander duality in a sphere, as in \cite[Theorem 2.4]{sta:gap}, since the order complex $\Omega_m$ is  homotopy equivalent to the sphere $\mathbb{S}^{m-2}$.
  \end{proof}

 Stanley's   Theorem~\ref{thm:sd-rs}, specialised to the rank-selected homology of $\Omega_{dm}^{(d)}$, appears below. We will use this version  as well as Theorem~\ref{rksel-Omd}  in our study of the multiplicity of the trivial representation in the homology, in the next section.

\begin{prop}[Special case of Theorem~\ref{thm:sd-rs}]\label{prop:d-div-Om} Let $T\subseteq [m-1]$. The corank-selected homology representation of $\Om_{dm}^{(d)}$ is given by the formula
\[\beta_{\Om_{dm}^{(d)}}(T)=\bigoplus_{i=0}^{m-1} c_{i,m}(T)\, \eta_i(\cB_{dm}^{(d))}).
\]
\end{prop}

We list below the first few \emph{corank}-selected homology representations $\beta_m(T)$ for $\Om_m$, up to $m=5,$ in the basis of homogeneous symmetric functions.  This data is computed from the recurrence in Theorem~\ref{rksel-Omd}. In view of Corollary~\ref{cor:Om-to-Omd}, these suffice to determine the corresponding representations for $\Om_{dm}^{(d)}$ for all $d$. Clearly $\beta_m(\emptyset)=h_m$. 
 By Proposition~\ref{Alex*-Om}, we need only list the representations for  half of the corank subsets $T$ of $[m-1]$.

\medskip
\noindent
For $m=4$ it suffices to list all the singleton corank sets:
%
\[\beta_4(\{1\}=2 h_3 h_1 +h_2^2 -h_4\qquad\qquad 
\beta_4(\{2\})=3 h_2 h_1^2-h_4  \qquad\qquad \beta_4(\{3\})=h_1^4-h_4
\]
For $m=5$: 
\begin{align*}
\beta_5(\{1\})&=2 h_3 h_2 +2 h_4 h_1  -h_5  &\quad \beta_5(\{1,2\})&=3 h_3h_1^2 +3 h_2^2 h_1 -2 h_4 h_1 - 2h_3 h_2 +h_5
\\
\beta_5(\{2\})&=3 h_3h_1^2+3 h_2^2 h_1 -h_5 &\quad \beta_5(\{1,3\})&=
8 h_2h_1^3 -2h_4 h_1 -2h_3 h_2 +h_5
\\
\beta_5(\{3\})&=4 h_2h_1^3 -h_5  &\quad \beta_5(\{1,4\})&=
3 h_1^5-2 h_4 h_1 -2 h_3 h_2 +h_5
\\
\beta_5(\{4\})&= h_1^5-h_5  
\end{align*}

\section{Multiplicity of trivial representation}
\label{mtr}

In this section we examine two enumerative invariants that arise in rank-selection for $\Om_m$, namely, the multiplicities of  the trivial representation for the actions of $\mathfrak{S}_m$ and $\mathfrak{S}_{m-1}$.  Here we are viewing $\fS_{m-1}$ as the subgroup of $\fS_m$ which fixes $m$.
Our motivation  is the case of the analogous numbers in the unordered partition lattice $\Pi_n$, which  are known to refine the Euler numbers \cite{sta:gap}, \cite{su:AIM1994}.   A  systematic study was undertaken in  \cite{su:AIM1994}, giving results and conjectures about their positivity.  Some of these were subsequently resolved by Hanlon and Hersh in \cite{HH:2003}, using a partitioning of the quotient complex and spectral sequences.  A complete list of currently known results for  $\Pi_n$ appears in \cite[Theorem 2.12]{su:RPS70-2016}.

 For $\Pi_n$, the question of determining the multiplicity of the trivial representation in the rank-selected homology is a difficult one.  Stanley had originally raised this question in \cite{sta:gap}. The first vanishing result is due to Hanlon \cite{Ha:1983}   and additional results were given by Sundaram, e.g. \cite[Proposition 3.4, Theorems 4.2-4.3 and 4.7]{su:AIM1994}. 
 By contrast, Theorem~\ref{TrivRepCom}  below determines this multiplicity completely for the $\mathfrak{S}_m$-action on corank-selected subposets of $\Omega_m$ and $\Omega_{dm}$.
 
 Inspired by Stanley's question, the paper \cite{su:AIM1994} also examined the multiplicity of the trivial representation for the action of $\mathfrak{S}_{n-1}$ on the rank-selected homology of $\Pi_n$, and showed that it exhibits interesting enumerative properties. 
 Again in contrast to the situation for $\Pi_n$,
 for $\Om_m$  and $\Omega_{dm}$ we are able to give a complete determination of this restricted multiplicity in Theorem~\ref{thm:triv-restrict-zeros}.

Consider $\Om_{dm}^{(d)}$ and a subset of coranks $T\subseteq [m-1]$ as in Theorem~\ref{rksel-Omd}.  Let $b_m(T)$ denote the multiplicity of the trivial representation of $\mathfrak{S}_{dm}$ in the homology of the corank-selected subposet $\Om_{dm}^{(d)}(T)$.  From~\eqref{RkSelHom} of Theorem~\ref{rksel-Omd} 
and the fact that $\#\cC(m,i+1)=\binom{m-1}{i}$, 
one sees  that  these numbers satisfy the recurrence
\begin{equation}\label{eqn:triv-rep-rec}
b_m(T) +b_m(T\setminus\{t_r\}) = \delta(T) \binom{m-1}{t_r}
\end{equation}
with initial condition  $b_m(\emptyset)=1$.
Moreover, since the  quantities $\delta(T)$ depend only on the subset $T$ of $[m-1]$, the invariants $b_m(T)$ are independent of $d$.  We therefore  assume without loss of generality that  $d=1$. 

From ~\eqref{alpha-beta} and Proposition~\ref{MaxChains}, we see that the numbers $\{b_m(T):T\subseteq [m-1]\}$ refine the factorials $(m-1)!$ 
\begin{equation}
\label{eqn:triv-rep-refinement}
\sum_{T\subseteq [m-1]} b_m(T)  = (m-1)!.
\end{equation}

We have the following combinatorial description of the multiplicities $b_m(T)$.
Note that part (b) can be expressed as in part (a) since the number of $\si\in\fS_{m-1}$ with $m-1$ in its descent set is zero.  These results have a topological explanation that we discuss at the end of this section.
\begin{thm}\label{TrivRepCom} 
Let $T$ be a subset of the nontrivial coranks of $\Om_m$.
\begin{enumerate}
\item[(a)] If $m-1\notin T$, then 
$$
b_m(T)=\#\{\sigma\in \mathfrak{S}_{m-1}: \Des\sigma=T\}.
$$
Hence $b_m(T)\ge 1$ in this case.
\item[(b)]
If $m-1\in T$, then $b_m(T)=0$. 
%
\end{enumerate}
\end{thm}
\begin{proof} Both parts are obtained from Proposition~\ref{prop:d-div-Om}, by putting $d=1$. This gives, for $\Om_m$, 
$$
\beta_{\Om_{m}}(T)=\bigoplus_{i=0}^{m-1} c_{i,m}(T)\, \eta_i(\cB_{m}).
$$
Recall from Theorem~\ref{BooleanRs} that $\eta_i(\cB_{m})$ is the sum of $\mathfrak{S}_m$-irreducibles indexed by ribbons $\rho_T$ for all $T$ of size $i$. 
The multiplicity of the trivial representation in the representation indexed by the ribbon $\rho_T$ is nonzero if and only if $\rho_T$ consists of a single row, and hence the nonzero contribution comes from the term $i=0$ in the  above sum. Thus we have 
\beq
\label{bc}
b_m(T)= c_{0,m}(T)=\#\{\sigma\in \mathfrak{S}_m: \sigma(m)=m \text{ and }\Des\sigma=T\}.  
\eeq

Now for (a), 
 observe that for any subset $T=\{i_1<\cdots<i_k\}\subseteq [m-2]$, it is easy to exhibit a permutation $\sigma\in \fS_m$ with $\si(m)=m$ and  descent set $T$. For example, let $\sigma(i_j)=m-j$, $1\le j\le k$, and, in one-line notation, fill the remaining slots of $\sigma$ with the letters in $[m-1]\setminus T$ in increasing order from left to right. 
As far as part  (b),  note that if $m-1\in T$ then $m-1$ would have to be a descent of $\si$.  But this is impossible since $\si(m)=m$.
\end{proof}

It is interesting to compare the  symmetry result below with the  sign-twisted symmetry of the homology representations resulting from Alexander duality, Proposition~\ref{Alex*-Om}.  In particular,   $b_m(T)$ is also the multiplicity of the \emph{sign} representation in $\beta_m([m-1]\setminus T)$. 

\begin{cor}
\label{TrivRepSym} 
Let $T$ be a subset of the nontrivial coranks of $\Om_m$.
The multiplicity of the trivial representation in  $\tilde{H}(\Omega_m(T))$ is nonzero if and only if $m-1\notin T$. Moreover, if $T\sbe[m-2]$ then one has the following symmetry:
\[b_m(T)=b_m([m-2]\setminus T).\]
\end{cor}
\begin{proof} The first statement is immediate from the two parts of Theorem~\ref{TrivRepCom}.  

For the second statement, given $\sigma\in \mathfrak{S}_{m-1}$, let $\tau$ be the permutation defined by letting $\tau(i)=m-\sigma(i)$.  
 The map $\sigma\mapsto \tau$  is clearly a bijection.  It also has the property that 
 $\Des\tau=[m-2]\setminus \Des\sigma$.  The result now follows from Part (a) of  Theorem~\ref{TrivRepCom}.  
\end{proof}

Next we consider the action of $\mathfrak{S}_{dm-1}$ on $\Om_{dm}^{(d)}$.  We follow the notation in \cite{su:AIM1994}.  
Let $b_m'(T)$ denote the multiplicity of the trivial representation of $\mathfrak{S}_{dm-1}$ on the corank-selected homology $\tilde{H}(\Om_{dm}^{(d)}(T))$, for every subset $T$ of $[m-1]$. 

From~\eqref{RkSelHom}   
one sees  that  these numbers satisfy the  initial condition  $b_m'(\emptyset)=1$ and the recursion 
\begin{equation}
\label{eqn:triv-rep-rec-restrict}
b_m'(T) +b_m'(T\setminus\{t_r\}) = \delta(T) (t_r+1)\binom{m-1}{t_r},
\end{equation}
since the restriction of $h_{d\alpha}$ to $\mathfrak{S}_{dm-1}$ is 
$\sum_{i=1}^{\ell(\alpha)}  h_{d\alpha_i-1}\,(\prod_{j\ne i} h_{d\alpha_j})$, and the multiplicity of the trivial representation in each term of the previous sum is 1. 
Again these numbers are independent of $d$, so that $b_m'(T)$ is also the multiplicity of the trivial representation of $\mathfrak{S}_{m-1}$ on the corank-selected homology $\tilde{H}(\Om_{m}(T))$, for every subset $T$ of $[m-1]$. 
Thus, as in the case of $b_m(T)$, it suffices to analyse the case $d=1$.

From ~\eqref{alpha-beta} and Proposition~\ref{MaxChains}, we see that the numbers $\{b_m'(T):T\subseteq [m-1]\}$  give the following refinement of the factorials $m!$
\begin{equation}\label{eqn:triv-rep-restrict-refinement}
\sum_{T\subseteq [m-1]} b_m'(T)  = m!
\end{equation}

Let $\gamma_m(T)=b_m'(T)-b_m(T)$. By Frobenius reciprocity, $\gamma_m(T)$ is the multiplicity of the irreducible indexed by $(m-1,1)$ in $\tilde{H}(\Om_{m}(T))$, and hence it is nonnegative.  In other words, $b_m'(T)\ge b_m(T)$. From~\eqref{eqn:triv-rep-rec} and~\eqref{eqn:triv-rep-rec-restrict}  we have the recurrence 
\begin{equation}\label{eqn:refl-rep-rec}
\gamma_m(T) +\gamma_m(T\setminus\{t_r\}) =\delta(T)\, t_r\binom{m-1}{t_r},
\end{equation}
with $\gamma_m(\emptyset)=0.$
Since $ t_r\binom{m-1}{t_r}=
(m-1) \binom{m-2}{t_r-1}$, this shows that $\gamma_m(T)$ is divisible by $(m-1)$ for all $T$.  

In \cite{sta:gap} and \cite[\S 4]{su:AIM1994}, it was shown that for the rank-selected homology of the unordered partition lattice $\Pi_n$, the multiplicities of the trivial representation of $\mathfrak{S}_n$ and  $\mathfrak{S}_{n-1}$  refine the Euler numbers $E_{n-1}$ and $E_n$, respectively, as sums over subsets of $[n-2]$.  
Thus equations~\eqref{eqn:triv-rep-refinement} and~\eqref{eqn:triv-rep-restrict-refinement} are the corresponding  analogues for $\Om_m$, respectively refining the factorials $(m-1)!$ and $m!$ by subsets of $[m-1]$.  

For $\Pi_n$ it was conjectured in \cite[Conjectures, p. 289]{su:AIM1994}, and proved in \cite[Theorems 2.1, 2.2]{HH:2003}, that the restricted multiplicity $b_n'(S)$ is always positive. Table~\ref{tab:trivpin-ord}
 shows that this is not true of the restricted multiplicities $b_m'(T)$ for $\Om_m$.  In fact, the data suggests that 
$b_m'(T)=0$ if and only if 
$T$ contains both the coranks $m-2, m-1$ and this is the case.
Indeed, we have the following theorem.  
\begin{rem} Stanley's theorem on barycentric rank-selection gives a different formulation of the result below.  We obtained the formulas below using the Whitney homology recurrence of Theorem~\ref{rksel-Omd}, but the proof is simpler using Stanley's theorem.
\end{rem}

\begin{table}[htbp!]
\begin{center}
\begin{tabular}{|c|c|c|}
\hline
$m, T$ & $\mathbf{b_m(T)}$& $\mathbf{b_m'(T)}$ \\
\hline
$T\subseteq [m-1]$ &  &    \\
\hline
\hline
    3, $\emptyset$ & 1 & 1\\
    3, \{1\} & 1 & 3 \\
    3, \{2\} & 0 & 2 \\
    3, \{1,2\} & 0 & 0\\[5pt]
    \hline\hline
    4, $\emptyset$ & 1 & 1\\
    4, \{1\} & 2 & 5 \\
    4, \{2\} & 2 & 8 \\
    4, \{3\} & 0 & 3\\
    4, \{1,2\} & 1& 4 \\
    4, \{1,3\} & 0 & 3 \\
    4, \{2,3\} & 0 & 0 \\
    4, \{1,2,3\} & 0 & 0 \\[5pt]
    \hline\hline
    5, $\emptyset$ & 1 & 1\\
    5, \{1\} & 3 & 7 \\
    5, \{2\} & 5 & 17 \\
    5, \{3\} & 3 & 15 \\
    5, \{4\} & 0 & 4 \\
    5, \{1,2\} & 3 & 11 \\
    5, \{1,3\} & 5 & 25 \\
    5, \{1,4\} & 0 & 8 \\
    5, \{2,3\} & 3 & 15\\
    5, \{2,4\} & 0 & 8 \\
    5, \{3,4\} & 0 & 0 \\
    5, \{1,2,3\} & 1 & 5 \\
    5, \{1,2,4\} & 0 & 4 \\
    5, \{1,3,4\} & 0 & 0 \\
    5, \{2,3,4\} & 0 & 0 \\
    5, \{1,2,3,4\} & 0 & 0 \\
    \hline
\end{tabular}
\quad 
\begin{tabular}{|c|c|c|c|}
\hline
$m, T$ & $\mathbf{b_m(T)}$& $\mathbf{b_m'(T)}$ \\
\hline
$T\subseteq [m-1]$ &  &    \\
\hline
\hline
    6, $\emptyset$ & 1 & 1\\
    6, \{1\}   & 4 & 9 \\
    6, \{2\}   & 9 & 29 \\
    6, \{3\}   & 9 & 39 \\
    6, \{4\}   & 4 & 24 \\
    6, \{5\}   & 0 & 5 \\
    6, \{1,2\} & 6 & 21 \\
    6, \{1,3\} & 16 & 71 \\
    6, \{1,4\} & 11 & 66 \\
    6, \{1,5\} & 0 & 15 \\
    6, \{2,3\} & 11 & 51 \\
    6, \{2,4\} & 16 & 96 \\
    6, \{2,5\} & 0 & 25 \\
    6, \{3,4\} & 6 & 36 \\
    6, \{3,5\} & 0 & 15 \\
    6, \{4,5\} & 0 & 0 \\
    6, \{1,2,3\} & 4 & 19 \\
    6, \{1,2,4\} & 9 & 54 \\
    6, \{1,2,5\} & 0 & 15 \\
    6, \{1,3,4\} & 9 & 54 \\
    6, \{1,3,5\} & 0 &  25\\
    6, \{1,4,5\} & 0 & 0 \\
    6, \{2,3,4\} & 4  & 24  \\
    6, \{2,3,5\} & 0 & 15 \\
    6, \{2,4,5\} & 0 & 0 \\
    6, \{3,4,5\} & 0 & 0 \\
    6, \{1,2,3,4\} & 1  & 6 \\
    6, \{1,2,3,5\} & 0 & 5 \\
    6, \{1,2,4,5\} & 0 & 0 \\
    6, \{1,3,4,5\} & 0 &  0\\
    6, \{2,3,4,5\} & 0 & 0 \\
    6, \{1,2,3,4,5\} & 0 & 0 \\
    \hline
\end{tabular}
\end{center}
\caption{The multiplicities $b_m(T)$ and $b_m'(T)$ by corank subsets $T$, for $3\le m\le 6$.\\ The second column adds up to $(m-1)!$ and the third column to $m!$. }
\label{tab:trivpin-ord}
\end{table}

\begin{thm}\label{thm:triv-restrict-zeros}  Let $T$ be a set of coranks, $T\subseteq [m-1]$.Then $b_m'(T)=0$  if and only if $T$ contains both the coranks $m-2, m-1$.  
More precisely, we have the following: 
\begin{enumerate}
\item[(a)] If $\{m-2, m-1\}\subseteq T$, then $b_m'(T)=0.$
\item[(b)] If $m-1\in T$, $m-2\notin T$, then $b_m'(T)\ge 1$.  In fact 
$$
b_m'(T)=\ga_m(T)=(m-1) b_{m-1}(T\setminus\{m-1\}).
$$
\item[(c)] If $m-1\notin T$, then $b_m'(T)\ge 1$.  In fact,
$$
b_m'(T)=mb_m(T)-(m-1)b_{m-1}(T).
$$
\end{enumerate}
\end{thm}

\begin{proof} 
Since 
\beq\label{b'gab}
b_m'(T)= \gamma_m(T)+b_m(T)
\eeq
we wish to determine $\gamma_m(T)$. 
This is the multiplicity of the irreducible indexed by $(m-1,1)$ in the right-hand side of Proposition~\ref{prop:d-div-Om} when $d=1$
\beq
\label{be_Om_m}
\beta_{\Om_{m}}(T)=\bigoplus_{i=0}^{m-1} c_{i,m}(T)\, \eta_i(\cB_{m}).
\eeq
The multiplicity of $(m-1,1)$ in the irreducible indexed by a ribbon $\rho(T)$ is nonzero if and only if the ribbon has only two rows, in which case it is 1. There are $m-1$ such ribbons of size $m$, and hence the multiplicity of $(m-1,1)$ in~\eqref{be_Om_m}
equals 
$$
(m-1)\,c_{1,m}(T)=(m-1)
 \#\{\sigma\in \mathfrak{S}_m: \sigma(m)=m-1\text{ and } \Des\sigma=T\},
 $$
and
\beq
\label{b'bc}
b_m'(T)=b_{m}(T)+ (m-1)\,c_{1,m}(T).
\eeq

If $m-1\in T$ then, by Theorem~\ref{TrivRepCom}  (b), we have $b_m(T)=0$.  We also see that when $m-1\in T$ and $\sigma\in \mathfrak{S}_m$ we have 
\[ \sigma(m)=m-1\text{ and } \Des\sigma=T\iff 
\sigma(m)=m-1, \sigma(m-1)=m \text{ and } \Des\sigma=T.
\]
Hence 
$$
c_{1,m}=\# \{ \sigma\in \mathfrak{S}_{m-2}:  \Des\sigma=T\setminus\{m-1\}\} .
$$
So, $c_{1,m}= 0$ if and only if $m-2\in T$.
Substituting all this information into~\eqref{b'gab} proves both (a) and (b).

For (c), suppose $m-1\notin T$.  It follows that $T\subseteq [m-2]$.  We claim that 
\beq
\label{cbb}
c_{1,m}(T) + b_{m-1}(T)
= b_m(T).   
\eeq
First note that if $\si$ is counted by $c_{1,m}$ then $\si(m)=m-1>\si(m-1)$ since $m-1\notin T$.
So, by definition, $c_{1,m}$ 
 counts the number of bijections $\tau:\{1, \ldots,m-2\} \rightarrow 
\{1, \ldots, m\}\setminus \{\sigma(m-1), m-1\}$ such that $\tau(j)>\tau(j+1)$ if and only if $j\in T$. 

On the other hand, by equation~\eqref{bc}, $b_m(T)$ counts the number of permutations $\sigma$ in $\mathfrak{S}_m$ with $\si(m)=m$ and $\Des\si=T$.   This set can be decomposed according to the image of $m-1$. By the preceding paragraph, 
the permutations with $\si(m-1)<m-1$ are counted by $c_{1,m}(T)$.  
If $\si(m-1)=m-1$ then the permutations are counted by $b_{m-1}$.  This completes the proof of~\eqref{cbb}.

Combining~\eqref{b'bc} and~\eqref{cbb} gives the equation in part  (c) of the theorem.  As far as the inequality, we have $b_m'(T)\ge b_m(T)$
from~\eqref{b'bc}.  Also, by Theorem~\ref{TrivRepCom} (a), we have $b_m(T)\ge1$ in this case.  Combining the two inequalities finishes the proof.
\end{proof}

We now give the promised topological explanation for Theorem~\ref{TrivRepCom}.  The reader may have noticed that the multiplicities $b_m(T)$, $T\subseteq [m-1]$,  coincide with the rank-selected Betti numbers for the Boolean lattice $\cB_{m-1}$. This is no accident, as we now explain.  We refer to \cite{HH:2003} for some background on quotient complexes and to \cite[\S 3.13]{sta:ec1} for flag $f$- and $h$-vectors.  

Recall that we write $P^*$ for the dual of the poset $P$.  If $P$ has a $\zh$ and a $\oh$ and $G$ is a group of automorphisms of $P$, the quotient complex $\Delta(P)/G$ consists of  the $G$-orbits of the faces of $\Delta(P)$, i.e., the $G$-orbits of the chains of the proper part $\Pb$. 
See, e.g.,  \cite[p. 522]{HH:2003}.   Furthermore, the multiplicity of the trivial representation of $G$ in the rank-selected homology of $P(T)$ for a rank-set $T$ is given by the flag $h$-vector $h_T(\Delta(P)/G)$ of the quotient complex \cite[p. 523]{HH:2003}.

Define the orbit poset $P/G$ as follows.  Its elements are  the $G$-orbits $\mathcal{O}_x$ of the elements $x$ of $P$, with order relation  $\mathcal{O}_x<\mathcal{O}_y$ in $P/G$ if there exist $x'\in \mathcal{O}_x$ and $y'\in \mathcal{O}_y$ 
such that $x'<y'$
in $P$.  
The quotient complex $\Delta(P)/G$ does not usually coincide with the order complex $\Delta(P/G)$ of the orbit poset $P/G$. In general the quotient complex may not even be a simplicial complex; an example is the ordinary partition lattice $\Pi_n$ with the $\fS_n$-action  \cite[p. 522]{HH:2003}.
But when $P=\Om_{dm}$ the quotient complex is not only a simplicial complex, but  also equals the order complex of the orbit poset, as we show in Proposition~\ref{DeOm-quot} below.

Let $ [\cB_{m-1}]$ denote the face lattice of the  $(m-2)$-dimensional ball on $m-1$ vertices, or equivalently the $(m-2)$-dimensional simplex on $(m-1)$ vertices, with an artificially appended $\hat 1$.  Hence, as a poset, $ [\cB_{m-1}]$ consists of all $2^{m-1}$ subsets of $[m-1]$, with an additional $\hat 1$ appended.

The order complex of $ [\cB_{m-1}]$ is the barycentric subdivision of the  $(m-2)$-dimensional simplex. It is therefore contractible. We record the following observations.
\begin{enumerate}
\item 
 $ [\cB_{m-1}]$ is a ranked Cohen-Macaulay poset, of rank $m$. 
 \item 
 The proper part of $ [\cB_{m-1}]$ then has a unique maximal element, the set $[m-1]$ of all $m-1$ elements, and hence the order complex of $ [\cB_{m-1}]$ is contractible.
 \item For the same reason,  the order complex of any rank-selected subposet  $ [\cB_{m-1}](T)$ where $m-1\in T$ is also contractible. 
 \item The rank-selected subposet $ [\cB_{m-1}](T)$ for a rank-set $T$  \emph{not} containing rank  $m-1$ coincides with the rank-selected subposet 
$\cB_{m-1}(T)$ of $\cB_{m-1}$.  In particular, when $T=[m-2]$, $ [\cB_{m-1}](T)$ 
coincides with the Boolean algebra $\cB_{m-1}$.
\end{enumerate}
As an example we compute  the quotient complex of $\Delta(\Om_3)$ by the action of $\mathfrak{S}_3$. 
Figure~\ref{Om3}  illustrates $\Om_3$. The 6 ordered partitions at corank 1 
are of the form $(\{a,b\},c)$ or $(a, \{b,c\})$ where $\{a,b,c\}=[3]$, and hence fall into two $\mathfrak{S}_3$-orbits that we label $\{1\}$ and $\{2\}$ respectively. The atoms  at corank 2 constitute a single orbit that we label $\{1,2\}$. The twelve chains between coranks 1 and 2 also fall into two orbits, corresponding to the two edges in Figure~\ref{fig:quot-complex-Om3}.  Hence the quotient complex looks like the one-dimensional simplicial complex shown in Figure~\ref{fig:quot-complex-Om3}. This is precisely the order complex of the (dual of the) poset $[\cB_2]$.  Note that coranks in $\Om_3$ correspond to ranks in $[\cB_2]$.

\begin{figure}
        \centering
        \begin{tikzpicture}
        \node (a) at (-2,-2) {$\{1\}$};
        \node (b) at (2,-2) {$\{2\}$};
        \node (min) at (0,-2) {$\{1,2\}$};
        \draw (a)--(min)--(b);
    \end{tikzpicture}
\caption{The quotient complex $\Delta(\Om_3)/\mathfrak{S}_3$}
 \label{fig:quot-complex-Om3}
 \end{figure}

 The proof of the next result highlights  the precise mapping between $\mathfrak{S}_{dm}$-orbits of chains in the dual of $\Om_{dm}$, and chains in $[\cB_{m-1}]$.
It will be convenient to use compositions $\al=(\al_1,\ldots,\al_r)$ of $m$.
Let $C_m$ denote the set of compositions of $m$.  For two compositions $\alpha, \beta \in C_m$, we say $\beta$ is a refinement of $\alpha=(\alpha_1, \ldots, \alpha_r)$ if $\beta$ is obtained from $\alpha$ by replacing each $\alpha_i$ with a composition of $\alpha_i$.  For example, the composition $(1,2,1,1,3,2)$ is a  refinement of the composition $(1,4,5)$.  With respect to this order,  $C_m$ is a poset with minimal element ${(1^m)}$, where $(1^m)$ is the composition of all 1's, and maximal element $(m)$.  There is a well-known bijection (see \cite[Theorem 1.7.1]{sag:aoc}) mapping compositions with $r$ parts  to subsets of $[m-1]$ of size $r-1$, namely 
\begin{equation}\label{eqn:comp-to-subsets}
\alpha=(\alpha_1, \ldots,\alpha_r)\mapsto \{\alpha_1, \alpha_1+\alpha_2, 
\ldots, \sum_{j=1}^{r-1} \alpha_j\}.
\end{equation}
 This makes the poset $C_m$  isomorphic to the dual of the Boolean algebra $\cB_{m-1}$, with the minimal composition $(1^m)$ mapping to the maximal subset $[m-1]$, and the maximal composition $(m)$ mapping to the empty set.

Finally, define $[C_m]$ to be the  poset $C_m$ augmented with an additional $\hat 0$ appended, i.e., 
$$
[C_m]=C_m\cup\{\hat 0\}.
$$
Thus $[C_m]$ is poset-isomorphic to the dual of $[\cB_{m-1}]$ as defined above.

\begin{prop}
\label{DeOm-quot} 
For $m\ge0$ and $d\ge1$ we have the following.
\ben
\item[(a)] There is an order-preserving isomorphism between the faces of the  quotient complex 
$\Delta(\Om_{dm}) /\mathfrak{S}_{dm}$ and the faces of  the order complex of the augmented composition poset $ [C_m]$.  
\item[(b)] The quotient complex 
$\Delta(\Om_{dm}^*) /\mathfrak{S}_{dm}$ is isomorphic  to the  order complex of $ [\cB_{m-1}]$. In particular, it is contractible.
\item[(c)] The quotient complex  $\Delta(\Om_{dm}) /\mathfrak{S}_{dm}$ coincides with the order complex $\Delta(\Om_{dm} /\mathfrak{S}_{dm})$ of the orbit poset.
\een

\end{prop}

\begin{proof} 
(a) We will show that the chains of the quotient complex $\Delta(\Om_{dm}) /\mathfrak{S}_{dm}$ are in order-preserving bijection with the chains in $C_m\setminus\{[m]\}$, the proper part of $[C_m]$.
We will do this when $d=1$.  The proof carries over to arbitrary $d$ almost verbatim, since compositions of $dm$ with all parts divisible by $d$ are in bijection with compositions of $m$.

Let $c=\omega_1<\omega_2<\cdots <\omega_r$ be a chain in the proper part of $\Om_m$, so $\omega_r\ne [m]$.  
Let $\omega_i=(B^i_1, \ldots, B^i_{t_i})$, with block sizes $b^i_1, \ldots, b^i_{t_i}$. 
Then the composition type $\beta^i=(b^i_1, \ldots, b^i_{t_i})$ of $\omega_i$ is a composition  of $m$ 
with $\be^i\neq [m]$ 
for all $i$, i.e., an element of $C_m\setminus\{[m]\}$.  Moreover the definition of the order relation in $\Om_m$ implies that 
$\beta^i< \beta^{i+1}$ in the poset of compositions $C_m$.  Note that if $\omega_1$ is an atom,  then its composition type is $\beta^1=(1^m)$. Now recall that the order complex of a poset consists of simplices which are chains in the \emph{proper} part of the poset.  Hence we have a surjection  $f: \Delta(\Om_m) \longrightarrow \Delta([C_m]).$  
Note that $f$ can be viewed as an extension of the bijection in the proof of Lemma~\ref{OmGen-d} (d) to chains, where $B_{m-1}$ takes the place of $C_m$.

This is clearly an order-preserving (and rank-preserving) surjection, and the equivalence classes under this surjection are precisely the $\fS_m$-orbits of $\Delta(\Om_m)$. 
This can be seen by observing that 
an orbit representative of a chain $c$ is uniquely determined, for example,  by writing the elements in each block of the bottom element of the chain in increasing order.  

It is best to illustrate this with an example.  In $\Om_9$, consider the chain 
$$
c=(1,2, 345,6, 7, 89)< (12, 3456, 789) < (123456, 789).
$$
It maps to the following chain of compositions in $[C_9]$
\[f(c)= (1,1,3,1,1,2)< (2,4,3)<(6, 3)\]
Any other chain in the preimage of the  chain $f(c)$  has the form 
$$
(a_1, a_2, a_3a_4a_5, a_6, a_7, a_8a_9)
<(a_1a_2, a_3a_4a_5, a_6, a_7a_8a_9)
<(a_1a_2a_3a_4a_5a_6, a_7a_8a_9).
$$
Hence the permutation $\sigma$ defined by $\sigma(i)=a_i$ takes the chain 
$c$ to the chain $c'$ in $\Delta(\Om_9)$.

We have shown that the preimage of the chain $f(c)$ is in fact the $\fS_m$-orbit of the chain $c$.
It follows that the quotient complex $\Delta(\Om_m)/\fS_m$ is homotopy equivalent to $\Delta([C_m])$. 

\medskip

(b) Combining (a) with the bijection~\eqref{eqn:comp-to-subsets} between $C_m$ and $\cB_{m-1}$, it follows that coranks in $\Om_m$ correspond to ranks in $[\cB_{m-1}]$.  Hence we have the poset isomorphism
\[ \Delta(\Om_m^*)/\fS_m\iso\Delta([\cB_{m-1})].\]

\medskip

(c)
The  demonstration of (a)  shows that the orbit poset $\Om_{dm}/\fS_{dm}$ coincides with the composition poset $[C_m]$.   
This finishes the proof of the proposition.
\end{proof}


By the discussion preceding the proposition,  this implies that the  multiplicities $b_m(T)$, $T\subseteq [m-1]$, of the trivial representation in 
the homology of the \emph{corank}-selected subposet $\Om_{dm}(T)$  coincide with the  flag $h$-vector (indexed by \emph{ranks}) of $\Delta([\cB_{m-1}])$. But, from Theorem~\ref{BooleanRs}, 
when $m-1\notin T$, the latter are precisely the rank-selected invariants  for the Boolean lattice $\cB_{m-1}$.  Again, see \cite[\S 3.13]{sta:ec1} for details.  Similarly, when $m-1\in T$, the multiplicity $b_m(T)$ is zero because the corresponding rank-selected subposet in $[\cB_{m-1}]$ is contractible. 
This concludes our topological explanation of Theorem~\ref{TrivRepCom}. 

Finally we remark that Theorem~\ref{thm:triv-restrict-zeros} determines the flag $h$-vector   for the quotient complex $\Delta(\Om_{m})/(\mathfrak{S}_{m-1}\times \mathfrak{S}_1)$. 
We have not investigated the structure of this quotient complex; it would be interesting to do so.    See \cite[\S 4]{HH:2003} for the analogous analysis for $\Pi_n$.


\section{Block sizes with remainder $1$}
\label{bsr}

Rather than just considering ordered set partitions where all blocks have size divisible by $d$, one could look at those where each block size has remainder $1$ when divided by $d$.  
Such ordered partitions of $n$ are partially ordered as a subset of $\Om_n$.  Adding a $\zh$, we get the poset
$$
\Omb_n^{(d)} =
\{\zh\}\ \uplus\
\{\om=(B_1,\ldots,B_k) \comp [n]\ 
\mid\ \#B_i \Cong 1\ (\Mod d) \text{ for all $1\le i \le k$}\}.
$$
Note that $\Omb_n^{(1)} = \Om_n$.

In this section, we will analyze this poset.  
To do this, we need a generalization of the Boolean algebra.
A {\em run} in a set $S\sbe[n]$ is a maximum sequence of consecutive integers.  For example, the runs in 
$S=\{2,3,4,6,8,9\}$ are $2,3,4$; $6$; and $8,9$.
Let
$$
\cB_n^{(d)}=\{S\sbe [n] \mid 
\text{every run of $S$ has length divisible by $d$}\},
$$
ordered by inclusion of sets.

\begin{thm}
\label{Omn1dGen}
The poset $\Omb_{dm+1}^{(d)}$ satisfies the following.
\ben
\item[(a)]  It has $\oh =([dm+1])$.
\item[(b)]  Its atoms are the $\om$ with $\#B=1$ for all blocks $B$ of $\om$.  
\item[(c)]  Every $\om\in\Omb_{dm+1}^{(d)}$ has $dk+1$ blocks for some $k\ge0$.
\item[(d)]
It is ranked.  The rank and corank of  $\om=(B_1,\ldots,B_{dk+1})$ are  
$$
\rk\om = m-k+1 \qmq{and} \crk\om = k.
$$
In particular
$$
\rk \Omb_{dm+1}^{(d)} = m + 1.
$$
\item[(e)] For any  two ordered partitions $\psi,\om\in\Omb_{dm+1}^{(d)}$, $\psi<\om$, we have
$$
[\psi,\om]\iso \cB_{\rk(\psi,\om)}^{(d)}.
$$

\item[(f)] The poset $\Omb_{dm+1}^{(d)}$ is not a lattice for $d\ge 2$.
\een
\end{thm}
\bprf
(a)--(e)  These proofs  follow the same lines as in the demonstrations of Theorems~\ref{OmGen-d} and so are omitted.

\medskip

(f) We provide a counterexample, using the notation in~\eqref{[n]}.  Consider the elements of rank $2$ given by
$$
(1,[2,d+2],d+3,\ldots,dm+1) \qmq{and} 
(1,2,[3,d+3],d+4,\ldots,dm+1).
$$
Then these are covered by both 
$$
([1,2d+1],2d+2,\ldots,dm+1) \qmq{and} 
(1,[2,2d+2],2d+3,\ldots,dm+1)
$$ 
so they have no join.
\eprf

We will now show that the M\"obius function on intervals in $\Omb_{dm+1}^{(d)}$ not containing $\zh$ is given, up to sign, by a $k$-Catalan number as defined below.  This follows from Lemma 3A about lattice paths in Chapter 1 in Narayana's book~\cite{nar:lpc} for which he indicates two proofs, one by inclusion-exclusion and one using determinants.  In Stanley's text~\cite{sta:cn}, he asks for a proof of the same formula for an isomorphic poset in Problem A18.  This problem was solved using algebraic manipulations by Kim and Stanton~\cite{KS}.  We will give a combinatorial proof using a sign-reversing involution on lattice paths.  For simplicity, we will begin with the case $d=2$.  We will need the following lemma.
\ble
\label{OmCrk}
Let  $a\in\Omb_{2m+1}^{(2)}$ be an atom.  Then
$$
\#\{ \om\in[a,\oh]\ \mid\ \crk \om = k\} =\binom{m+k}{2k}.
$$
\ele
\bprf
If $\om\in\Omb_{2m+1}^{(2)}$ then every block of $\om$ has odd size.
And we know from part (d) of the previous theorem that $\om$ has an odd number of blocks.  We will first look at the number of compositions which could be the type of such an $\om$.

Fix $k$ and $m$ and consider the quantity
$$
A_{m,k} := \#\{\al=(\al_1,\ldots,\al_{2k+1})\ 
\mid\ \al\comp 2m+1 \text{ and all $\al_i$ are odd}\}.
$$
Letting $\be_i=\al_i+1$ for all $i$ we see that
$$
A_{m,k} =\#\{\be=(\be_1,\ldots,\be_{2k+1})\ 
\mid\ \be\comp 2m+2k+2 \text{ and all $\be_i$ are even}\}.
$$
Now defining $\ga_i = \be_i/2$ for $1\le i\le 2k+1$ gives
\beq
\label{Amk}
A_{m,k}= \#\{\ga=(\ga_1,\ldots,\ga_{2k+1})\ 
\mid\ \ga\comp m+k+1\} 
=\binom{m+k}{2k}.
\eeq

We now return to the problem of counting the number of $\om\in[a,\oh]$ of corank $k$.  By Theorem~\ref{Omn1dGen} we know that $\om$ has $2k+1$ blocks.
So its type is one of the compositions in the set defining $A_{m,k}$ above.
But each such type corresponds to a unique $\om$ in the interval because once the type is specified, the blocks must be filled consistent with the fact that $\om\ge a$.  It follows that the binomial coefficient in equation~\eqref{Amk} also enumerates the ordered set partitions in question.
\eprf

We now review the necessary facts about sign-reversing involutions and lattice paths.  Let $S$ be a finite set and let $\io:S\ra S$ be an involution, i.e., a bijection such that $\io^2 = \io$.  Considering $\io$ as a permutation of $S$, it can be decomposed into cycles.  And $\io$ is an involution if and only if each of these cycles has length $1$ (a fixed point), or length $2$.  Now suppose that $S$ is signed in that one is given a function $\sgn:S\ra\{+1,-1\}$.  We say that  $\io$ is {\em sign reversing} if
\ben
\item[(i1)]  For every fixed point $(s)$ we have $\sgn s = 1$, and
\item[(i2)]  For every $2$-cycle $(s,t)$ we have
$$
\sgn s = -\sgn t.
$$
\een
It follows that 
\beq
\label{IoSum}
\sum_{s\in S}\sgn s = \#\Fix\io
\eeq
where $\Fix\io$ is the set of fixed points of $\io$.  This is because, by (i1), all fixed points have positive sign.  And, by (i2), the summands corresponding to elements in $2$-cycles cancel in pairs.

\begin{figure}
    \centering
\begin{tikzpicture}
\draw(1,4) node{$M=$};
\filldraw(0,0) circle(.1);
\filldraw(0,1) circle(.1);
\filldraw(0,2) circle(.1);
\filldraw(1,3) circle(.1);
\filldraw(2,3) circle(.1);
\filldraw(3,4) circle(.1);
\filldraw(4,5) circle(.1);
\filldraw(4,6) circle(.1);
\filldraw(5,6) circle(.1);
\filldraw(6,6) circle(.1);
\draw (0,0)--(0,2)--(1,3)--(2,3)--(4,5)--(4,6)--(6,6);
\draw[->] (0,0)--(7,7);
\draw(7,7.5) node{$y=x$};
\end{tikzpicture}
\qquad
\begin{tikzpicture}
 \draw(-1.5,2) node{$M\da=$};
\filldraw(0,0) circle(.1);
\filldraw(0,1) circle(.1);
\filldraw(0,2) circle(.1);   
\filldraw(1,2) circle(.1);
\filldraw(1,3) circle(.1);
\filldraw(2,3) circle(.1);   
\filldraw(3,3) circle(.1);
\draw (0,0)--(0,2)--(1,2)--(1,3)--(3,3);
\draw[->] (0,0)--(4,4);
\draw(4,4.5) node{$y=x$};
\end{tikzpicture}
    \caption{The restriction of a Motzkin path}
    \label{Mres}
\end{figure}

A {\em lattice path} is a sequence of points $P:p_0,p_1,\ldots,p_k$ in the integer lattice $\bbZ^2$.  These points are connected by line segments called {\em steps}.  The {\em length} of $P$, $\ell(P)$, is the number of steps. We will use the following steps, called {\em north} (N), {\em east} (E), and {\em diagonal} (D), which go from $p=(x,y)$ to 
$(x,y+1)$, $(x+1,y)$, and $(x+1,y+1)$, respectively.  A {\em Motzkin path} is a lattice path, $M$, satisfying
\ben
\item[(m1)] $M$ starts at $(0,0)$ and ends at $(n,n)$ for some $n$, 
using steps $N$, $E$, and $D$,   and
\item[(m2)] $M$ never goes below the line $y=x$.
\een
We will often specify a Motzkin path by listing its sequence of steps.  In such a sequence, a consecutive pair of steps of the form $NE$ will be called a {\em corner}.
To illustrate,  the Motzkin path $M$ of Figure~\ref{Mres} has a unique corner formed by the seventh and eighth steps.
A {\em Dyck path} is a Motzkin path with no diagonal steps.
The {\em restriction} of a Motzkin path $M$ is the Dyck path $M\da$ obtained by removing all the diagonal steps of $M$ and concatenating what remains.  An example can be found in Figure~\ref{Mres} where
$M=NNDEDDNEE$ and $M\da = NNENEE$.

The last bit of notation we need is that for the {\em Catalan numbers}
$$
C_n = \frac{1}{n+1}\binom{2n}{n}.
$$
It is well known that $C_n$ is the number of Dyck paths $P$ ending at $(n,n)$, equivalently, those with $\ell(P)=2n$.

\begin{figure}
    \centering
 \begin{tikzpicture}
\filldraw(0,0) circle(.1);
\filldraw(1,1) circle(.1);
\filldraw(2,2) circle(.1);
\draw (0,0)--(2,2);
\draw[<->] (3,1)--(5,1);
 \end{tikzpicture}
 \qquad
\begin{tikzpicture}
\filldraw(0,0) circle(.1);
\filldraw(0,1) circle(.1);
\filldraw(1,1) circle(.1);
\filldraw(2,2) circle(.1);
\draw (0,0)--(0,1)--(1,1)--(2,2);
 \end{tikzpicture}   

\vs{20pt}

 \begin{tikzpicture}
\filldraw(0,0) circle(.1);
\filldraw(0,1) circle(.1);
\filldraw(1,1) circle(.1);
\filldraw(1,2) circle(.1);
\filldraw(2,2) circle(.1);
\draw (0,0)--(0,1)--(1,1)--(1,2)--(2,2);
\draw[<->] (3,1)--(5,1);
 \end{tikzpicture}
 \qquad
\begin{tikzpicture}
\filldraw(0,0) circle(.1);
\filldraw(1,1) circle(.1);
\filldraw(1,2) circle(.1);
\filldraw(2,2) circle(.1);
\draw (0,0)--(1,1)--(1,2)--(2,2);
 \end{tikzpicture}   

 \vs{20pt}

  \begin{tikzpicture}
\filldraw(0,0) circle(.1);
\filldraw(0,1) circle(.1);
\filldraw(0,2) circle(.1);
\filldraw(1,2) circle(.1);
\filldraw(2,2) circle(.1);
\draw (0,0)--(0,2)--(2,2);
\draw[<->] (3,1)--(5,1);
 \end{tikzpicture}
 \qquad
\begin{tikzpicture}
\filldraw(0,0) circle(.1);
\filldraw(0,1) circle(.1);
\filldraw(1,2) circle(.1);
\filldraw(2,2) circle(.1);
\draw (0,0)--(0,1)--(1,2)--(2,2);
 \end{tikzpicture}   
    \caption{The involution $\io$ of Theorem~\ref{muOmc:d=2} when $m=2$}
    \label{ioEx}
\end{figure}

\bth
\label{muOmc:d=2}
For any ordered set partitions $\psi\le \om$ in $\Omb_{2m+1}^{(2)}$ 
with $\rk(\psi,\om)=k$ we have
$$
\mu(\psi,\om) = (-1)^k C_k
$$
\eth
\bprf   By Theorem~\ref{Omn1dGen}, it suffices to prove the result for intervals of the form $[\psi,\oh]$.  We will induct on $\crk\psi$, where the base case is easy.  By induction, we can assume that $\psi$ is an atom.  Now $\mu(\psi,\oh)$ is uniquely defined as the solution to the equation
$$
\sum_{\om\in[\psi,\oh]}\mu(\om,\oh) = 0
$$
where the values $\mu(\om,\oh)$ for $\om>\psi$ are already known by induction.  Since the solution is unique, it suffices to prove that
\beq
\label{SignedCSum}
\sum_{\om\in[\psi,\oh]} (-1)^{\crk\om}\  C_{\crk\om} = 0.
\eeq
But using Lemma~\ref{OmCrk} we obtain
$$
\sum_{\om\in[\psi,\oh]} (-1)^{\crk\om}\  C_{\crk\om}
= \sum_{k=0}^m\ \sum_{\substack{\om\in[\psi,\oh] \\[2pt] \crk\om=k}} (-1)^k C_k
=\sum_{k=0}^m  (-1)^k \binom{m+k}{2k} C_k.
$$
The (unsigned) $k$th term in this last sum is just the number of Motzkin paths 
$M$ ending at $(m,m)$ with $m+k$ steps such that $M\da$ has length $2k$.  Indeed, start with the Motzkin path $M_0$ consisting of $m+k$ diagonal steps.  Now choose $2k$ of them to be replaced by a step $N$ followed by a step $E$.  Finally, given one of the  $C_k$ Dyck paths $P$,
there is a unique Motzkin path $M$ such that $M\da = P$ and the orthogonal projection of the $D$ steps of $M$ onto the line $y=x$ are the ones not chosen to be replaced in $M_0$.  
Also, note that the sign associated to $M$ is $(-1)^{\ell(P)/2}$.

 By~\eqref{IoSum} we will be done if we can find a sign-reversing involution, $\io$, without fixed points. Scan one of the Motzkin paths, $M$, in question from left to right until one finds the first occurrence of either a diagonal step or a pair of steps forming a corner.  Define $\io(M)=M'$ where $M'$ is formed by switching this first occurrence from a diagonal step to a corner or vice-versa.  See Figure~\ref{ioEx} for an example where the paths of positive sign are on the left and those of negative sign are on the right.  By the definition of the map, it is clearly an involution.  And it is also sign-reversing since the lengths of $M\da$ and $M'\da$ differ by $2$.  
 Since any nontrivial Motzkin path has either a diagonal step or a corner, $\io$ has no fixed points.
 This completes the proof.
\eprf

To describe the M\"obius function of intervals in $\Omb_{dm+1}^{(d)}$ not containing $\zh$ for all $d$, we need a generalization of the Catalan numbers.  For $n\ge0$ and $k\ge1$, the {\em $k$-Catalan numbers} are
$$
C_{n,k} = \frac{1}{(k-1)n+1}\binom{kn}{n}.
$$
Note that $C_{n,1}=1$ and  $C_{n,2} = C_n$.  The associated generating function
$$
C_k(x) = \sum_{n\ge0} C_{n,k} x^n
$$
satisfies the implicit relation
$$
C_k(x) = 1 + x (C_k(x))^k.
$$

To describe the associated lattice paths, define a {\em $k$-diagonal} step, $D^{(k)},$ to be one which goes from a lattice point $(x,y)$ to $(x+1,y+k-1)$.  So $D^{(2)} = D$.  A {\em $k$-Motzkin path} is defined by 
\ben
\item[(M1)] $M$ starts at $(0,0)$ and ends at $(n,(k-1)n)$ for some $n$,
using steps $N$, $E$, and $D^{(k)}$,   and
\item[(M2)] $M$ never goes below the line $y=(k-1)x$.
\een
A {$k$-corner} in such a path is $k-1$ steps $N$ followed by $1$ step $E$.
A {\em $k$-Dyck path} is a $k$-Motzkin path with no $k$-diagonal steps.  It is well known that the number of $k$-Dyck paths ending at $(n, (k-1)n)$ is $C_{n,k}$.
The {\em restriction} of a $k$-Motzkin path is defined similarly to the case $k=2$.
\bth
\label{thm:muOmb}
For all $d\ge1$ and  ordered set partitions $\psi\le \om$ in $\Omb_{dm+1}^{(d)}$ 
with $\rk(\psi,\om)=k$ 
\beq
\label{muOmb}
\mu(\psi,\om) = (-1)^k C_{k,d}.
\eeq
\eth
\bprf  The case $d=1$ follows from Theorem~\ref{OmndComb} and the fact that $\cE_n^{(1)}=(-1)^n$.
So, we assume that $d\ge2$.
Much of the demonstration parallels that of the case when $d=2$, so we will only provide details for the differences. Again, we need to prove equation~\eqref{SignedCSum} where now $\psi$ lies in $\Omb_{dm+1}^{(d)}$.
Using a $d$-analogue of Lemma~\ref{OmCrk}, we see that this is equivalent to
$$
\sum_{k=0}^m  (-1)^k \binom{(d-1)k+m}{dk} C_{k,d}=0.
$$
Now the terms in the sum count $d$-Motzkin paths ending at $(m,(d-1)m)$ with 
$(d-1)k+m$ steps such that $M\da$ has $dk$ steps.  The involution switches a $d$-diagonal step with a $d$-corner in the same way as when $d=2$.  The reader should now be able to fill in the details.
\eprf

We will now discuss the M\"obius values $\mu(\zh,\om)$ in $\Omb_{dm+1}^{(d)}$.  It is easy to see that if $\om=(B_1,\ldots,B_\ell)$ then we have the reduced product
\beq
\label{ZhOmbRp}
[\zh,\om]\iso \Omb_{\#B_1}^{(d)} \dot{\times}\cdots \dot{\times}\, \Omb_{\#B_\ell}^{(d)}.
\eeq
So, by Theorem~\ref{PxQ}, it suffices to compute $\mu(\Omb_{dm+1}^{(d)})$.
Motivated by Theorem~\ref{OmndComb} (b), let us define the {\em remainder $1$ Euler numbers} by
\beq
\label{cEbDef}
\cEb_{n}^{(d)}=\mu(\Omb_n^{(d)})
\eeq
where $\mu$ is zero if $n\not\Cong 1\ (\Mod d)$.  
This notation is extended to ordered set partitions $\om=(B_1,\ldots,B_\ell)$
by letting
$$
\cEb_\om^{(d)} = \cEb_{\#B_1}^{(d)}\cdots \cEb_{\#B_\ell}^{(d)}.
$$
On the generating function level, we let
$$
\cEb_d(x) = \sum_{n\ge1} \cEb_{n}^{(d)}\ \frac{x^n}{n!}.
$$
We will also need the series
$$
F_d(x) = \sum_{m\ge 0}\frac{x^{dm+1}}{(dm+1)!}.
$$

Other notation includes
$$
\om\compb_d\ [n]
$$
to indicate that $\om$ is an ordered set partition of $[n]$ where all parts have size  congruent to $1$ modulo $d$.  Finally we will need the round up and round down functions $\ce{\ell/d}$ and $\fl{\ell/d}$, respectively.
Here are some of the properties of the $\cEb_{n}^{(d)}$. 
\bth
\label{bOmndComb}
Suppose $d,n\ge1$ and $n\Cong 1\ (\Mod d)$.
\ben
\item[(a)] We have
$$
\cEb_{n}^{(d)} = \sum_{\om\compb_d \hs{2pt} [n]} 
(-1)^{\ce{\ell/d}}\ C_{\fl{\ell/d},d}
$$
where $\ell=\ell(\om)$.
\item[(b)] We have
$$
\cEb_{n}^{(d)} = -1+\sum_{\substack{\om\compb_d \hs{2pt} [n]
\\[4pt] 
\om\neq([n])}}
(-1)^\ell\ \cEb_\om^{(d)}
$$
where $\ell=\ell(\om)$.
\item[(c)]  We have
$$
\cEb_d(x) = -F_d(x)\cdot C_d( -F_d(x)^d ).
$$
\item[(d)]  We have
$$
\cEb_d(x) = -F_d(x)\cdot (1-(-1)^d\cEb_d(x)^d).
$$
\een
\eth
\bprf
(a)
Recall definitions~\eqref{cEbDef} and~\eqref{muDn}, equation~\eqref{muOmb}, as well as the fact that in $\Omb_n^{(d)}$ we have $\crk\om = \fl{\ell/d}$
where $\ell=\ell(\om)$.  Putting these together gives
\begin{align*}
\cEb_{n}^{(d)} &=\mu(\Omb_{n}^{(d)})\\
&= - \sum_{\om\compb_d \hs{2pt} [n]} \mu(\om,\oh)\\
&= - \sum_{\om\compb_d \hs{2pt} [n]} (-1)^{\fl{\ell/d}}\ C_{\fl{\ell/d},d}\\
\end{align*}
and bringing the negative sign inside the sum completes the proof of this part.

\medskip
(b)  Using definition~\eqref{cEbDef}, equation~\eqref{ZhOmbRp}, and Theorem~\ref{PxQ} we see that if an ordered set partition has $\type(\om)=(\al_1,\ldots,\al_\ell)$ then
$$
\mu(\zh,\om)
= (-1)^{\ell-1}\mu(\Omb_{\al_1}^{(d)})\cdots \mu(\Omb_{\al_\ell}^{(d)})
= (-1)^{\ell-1}\cEb_\om^{(d)}.
$$
Noting that this also applies to $\om=([n])$ and that $\mu(\zh,\zh)=1$, we can bring all the terms in the desired equality over to the left side and write it as
$$
\sum_{x\in\Omb_n^{(d)}} \mu(\zh,x) = 0.
$$
where $x$ runs over both the ordered set partitions and $\zh$ in $\Omb_n^{(d)}$.
But this is true by definition~\eqref{muUp}.

\medskip

(c)  Similar to the proof of Corollary~\ref{OmdGfs} we have, for any exponential generating function of the form $E(x) = c_1 x/1! + c_2 x^2/2!+ \cdots$,
\begin{align}
  E(x)^{\ell}
  &= \sum_{n\ge \ell} \left(\sum_{\al=(\al_1,\ldots\al_\ell)\comp n }
  \frac{c_{\al_1}x^{\al_1}}{\al_1!}\cdots\frac{c_{\al_\ell}x^{\al_\ell}}{\al_\ell!}    \right)
  \notag
  \\
  &= \sum_{n\ge \ell}\left(\sum_{\al\comp n }
  \binom{n}{\al} c_{\al_1}\dots c_{\al_\ell}   \right) \frac{x^n}{n!}
  \notag
  \\
    &= \sum_{n\ge \ell}\left(\sum_{\om=(B_1,\ldots, B_\ell)\comp [n] }
     c_{\#B_1}\dots c_{\#B_\ell}\right) \frac{x^n}{n!}
     \label{egf}
\end{align}
where the last equality comes from the fact that the number of ordered set partitions $\om$ with $\type\om =\al$ is the multinomial coefficent
$$
\binom{n}{\al}=\frac{n!}{\al_1! \cdots,\al_\ell!}.
$$

Write $n=dm+1$ and note that  if $\om\in\Omb_n^{(d)}$ then 
$\ell(\om)$ and the size of its blocks are congruent to $1$ modulo $d$.  Combining these facts with part (a) and the computations in the previous paragraph gives
$$
\cEb_d(x) 
= -C_{0,d} F_d(x) + C_{1,d}(x) F_d(x)^{d+1} - C_{2,d} F_d(x)^{2d+1}-\cdots.
$$
Factoring out $-F_d(x)$ and simplifying completes the proof of part (c).

\medskip

(d)  From part (b) we have
$$
0 = -1 + \sum_{\om\compb_d \hs{2pt} [n]}
(-1)^\ell\ \cEb_\om^{(d)}
$$
Let $n=dm+1$, multiply the previous displayed equation by $x^{dm+1}/(dm+1)!$,  sum over $m$, and use~\eqref{egf} to get
$$
0 = -F_d(x) + \sum_{m\ge0} (-1)^{dm+1}\cEb_d(x)^{dm+1}
= -F_d(x) -\frac{\cEb_d(x)}{1-(-1)^d\cEb_d(x)^d}.
$$
Solving for $\cEb_d(x)$ finishes the demonstration of (d) and of the theorem.
\eprf

One could hope that $\Omb_n^{(d)}$ has a recursive atom ordering.  Unfortunately, this is not the case, at least if one uses the same lexicographic order employed for $\Om_{dn}^{(d)}$ in the proof of Theorem~\ref{Om^d-RAO}.  
Consider what happens when $n=7$ and $d=2$.
By Theorem~\ref{Omn1dGen} (f), $\Omb_n^{(d)}$ need not be a lattice, much less semimodular.  So we need to show that condition (R1) in the definition of an RAO holds for the full poset
$\Omb_7^{(2)}$.
Consider the atom $a=(4,1,2,3,5,6,7)$.  Then the interval $[a,\oh]$ contains the atoms and $\psi=(4,1,2,653,7)$ and 
$\om=(421,3,5,6,7)$.  Note that lexicographically $\psi<_l \om$ since $4<_l 421$.  But $a$ is the smallest atom of 
$\Omb_7^{(2)}$ below $\psi$ since all such atoms are of the form $(4,1,2,a,b,c,7)$ where $abc$ is a permutation of $3,5,6$,  And
$356$ as it appears in $a$ is the lexicographically  smallest such permutation.  On the other hand, $\om$ covers $(1,2,4,3,5,6,7)$ which is lexicographically smaller than $a$.   So (R1) is violated.
But perhaps something can still be said, even for the weaker condition of shellability.
\begin{question}
Is $\Omb_n^{(d)}$ shellable?
\end{question}

It would be natural to consider posets derived from ordered set partitions of $n$ where all block sizes are congruent to $r$ modulo $d$ for $r>1$ and with a $\zh$ added.    Unfortunately, these posets do not seem to be well behaved.  In fact, they are not even graded.  For example, take $r=2$, $d=3$, and $n=14$.  Then, using the notation in~\eqref{[n]}, one maximal chain in this poset is 
$$
\zh < ([1,5],\ [6,10],\ [11,12],\ [13,14]) < \oh
$$
which is of length $2$.
On the other hand, the maximal chains containing the atom
$$
([1,2],\ [3,4],\ [5,6],\ [7,8],\ [9,10],\ [11,12],\ [13,14])
$$
all have length $3$.

\begin{question}
What is the M\"obius function of the poset of ordered partitions on $n$ with all block sizes congruent to $r>1$?  Is it shellable?  Does it have interesting homology representations?
\end{question}

\medskip

{\em Acknowledgements.}  We are very grateful to the anonymous referees for their generous investment of time and effort to provide exceptionally detailed and instructive reports. In particular, consideration of the reports led us to include Theorem~\ref{thm:sdP-rksel-recS} in our revisions. 

We wish to thank Jang Soo Kim and Richard Stanley for useful discussions.
The second author thanks Toufik Mansour for stimulating conversations.
We also thank Russ Woodroofe and Victor Reiner for pointing out relevant references and suggesting improvements.






\nocite{*}
\bibliographystyle{alpha}

\bibliography{2026Mar16REVospref}

\end{document}